\documentclass[11pt, reqno]{amsart}
\usepackage{amsmath, amsthm, amscd, amsfonts, amssymb, graphicx, color, mathtools}
\usepackage{mathrsfs}
\usepackage[bookmarksnumbered, colorlinks, plainpages]{hyperref}
\usepackage{enumerate}
\usepackage{setspace}
\usepackage{multicol}
\usepackage[margin=1in]{geometry}
\usepackage{comment}
\usepackage{dsfont}

\usepackage[normalem]{ulem}
\newcommand\tsout{\bgroup\markoverwith{\textcolor{red}{\rule[0.5ex]{2pt}{1.4pt}}}\ULon}
\newcommand{\stkout}[1]{\ifmmode\text{\tsout{\ensuremath{#1}}}\else\tsout{#1}\fi}
\allowdisplaybreaks

\theoremstyle{definition}
\newtheorem{theorem}{Theorem}[section]
\newtheorem{lemma}[theorem]{Lemma}
\newtheorem{proposition}[theorem]{Proposition}

\newtheorem{definition}[theorem]{Definition}

\newtheorem{remark}[theorem]{Remark}
\numberwithin{equation}{section}

\newcommand{\bb}{\mathbb}

\newcommand{\dt}{\mathrm{d}t}

\newcommand{\dx}{\mathrm{d}x}
\newcommand{\ds}{\mathrm{d}s}

\newcommand{\dW}{\mathrm{d}W}

\newcommand{\divg}{\mathrm{div}\,}
\newcommand{\curl}{\mathrm{curl}\,}

\newcommand{\norm}[1]{\left\|{#1}\right\|}

\newcommand{\inpro}[2]{\left\langle#1,#2\right\rangle}

\newcommand{\abs}[1]{\left|{#1}\right|}
\newcommand{\one}{\mathds{1}}

\def\be{\begin{equation}\label}
\def\ee{\end{equation}}
\def\bd{\begin{definition}\label}
\def\ed{\end{definition}}
\def\bt{\begin{theorem}\label}
\def\et{\end{theorem}}
\def\el{\end{lemma}}
\def\br{\begin{remark}\label}
\def\er{\end{remark}}
\def\bal{\[\begin{aligned}}
\def\eal{\end{aligned}\]}

\begin{document}
\setcounter{page}{1}

\title[Strong pathwise solutions for a class of stochastic TMHD-type systems]{Strong pathwise solutions for a class of stochastic thermo-magneto-hydrodynamic-type systems with multiplicative noise}

\author[Agus L. Soenjaya]{Agus L. Soenjaya}
\address{School of Mathematics and Statistics, The University of New South Wales, Sydney 2052, Australia}
\email{\textcolor[rgb]{0.00,0.00,0.84}{a.soenjaya@unsw.edu.au}}

\author[Thanh Tran]{Thanh Tran}
\address{School of Mathematics and Statistics, The University of New South Wales, Sydney 2052, Australia}
\email{\textcolor[rgb]{0.00,0.00,0.84}{thanh.tran@unsw.edu.au}}

\date{\today}

\begin{abstract}
    We establish the existence and uniqueness of strong solutions, in both the PDE and probabilistic sense, for a broad class of nonlinear stochastic partial differential equations (SPDEs) on a bounded domain $\mathscr{O}\subset \mathbb{R}^d$, $d\in \{1,2,3\}$, driven by multiplicative Gaussian noise. The solutions are global in time for $d\in \{1,2\}$. This theory simultaneously covers several physically relevant systems, including a class of stochastic generalised Burgers equations, a class of stochastic damped Navier--Stokes equations, stochastic magnetohydrodynamics (MHD), stochastic B\'enard convection in porous media, stochastic convective dynamo system, stochastic thermo-magneto-micropolar fluids, and stochastic diffusive tropical climate model, for which previous results only provide analytically weak martingale or pathwise solutions. The proof relies on Galerkin approximation and compactness argument. Up to a suitable stopping time, we derive strong moment bounds and verify a Cauchy property for the approximate solutions, in the absence of any inherent cancellation structure. By employing a localisation argument and a Gronwall-type lemma for stochastic processes, we establish the existence and uniqueness of maximal strong pathwise solutions, which are global when $d\in \{1,2\}$.
\end{abstract}
\maketitle


\section{Introduction}

This paper investigates the existence and uniqueness of \emph{strong} solutions, in the sense of PDEs and probability, for a class of abstract nonlinear stochastic partial differential equations (SPDEs) defined on a bounded domain $\mathscr{O}\subset \bb{R}^d$, where $d\in \{1,2,3\}$. These problems are motivated by thermally-coupled models in magneto-hydrodynamics {(MHD)}, which describe the interplay between thermal, magnetic, and fluid dynamic phenomena under stochastic influences. The problem takes this general form:
\begin{equation}\label{equ:equation}
	\begin{aligned}
	\frac{\mathrm{d} \Phi}{\mathrm{d}t}+ \mathcal{A} \Phi + \mathcal{B}(\Phi,\Phi) + \mathcal{R}(\Phi)
	&=
	g(\Phi) \dot{W},
	\\
	\Phi(0) &= \Phi_0,
	\end{aligned}
\end{equation}
which describes the time evolution of a state vector $\Phi$. Here, $\mathcal{A}$ is a linear operator representing dissipative effects, $\mathcal{B}$ is a bilinear form capturing convective nonlinearities, and $\mathcal{R}$ is a nonlinear operator accounting for other reactive physical effects. The term $g(\Phi) \dot{W}$ represents a multiplicative Gaussian noise modelling random forcing that depends on the state of the system. 

To cover a range of physically relevant thermo-magneto-hydrodynamical systems, a set of assumptions on the maps $\mathcal{A}$, $\mathcal{B}$, $\mathcal{R}$, and $g$ will be imposed later. These assumptions ensure that \eqref{equ:equation} encompasses several physically relevant models.
For ease of reference, we indicate below the section in which each example is discussed in detail:
\begin{enumerate}
\item Section~\ref{sec:scbf}: a class of stochastic damped Navier--Stokes equations~\cite{BrzFer19, CaiJiu08, GaoLiu19b, HajRob17, LiKim24};
\item Section~\ref{sec:burgers}: a class of stochastic generalised Burgers equations~\cite{BrzGolNek14, DaGat95, DonGuoWuZho23, DonWuZho24, Kim06};
\item Section~\ref{sec:mhd}: the stochastic magnetohydrodynamic (MHD) equations~\cite{Mot22, SerTem83, SriSun99, TitTra19, Yam16};
\item Section~\ref{sec:sbc}: the stochastic Boussinesq model for B\'enard convection~\cite{ChuMil10, DuaMil09, TitTra24, Yam16};
\item Section~\ref{sec:dynamo}: the stochastic convective dynamo system~\cite{CatEmoWei03, CatHug06, Tob21};
\item Section~\ref{sec:mmp}: the stochastic magneto–micropolar fluid system~\cite{InoMatOta03, KalLanLuk19, XuQiaZha21, Yam14, Yam16m}; and
\item Section~\ref{sec:climate}: the stochastic diffusive tropical climate model~\cite{FriMajPau04, LiTi16, Ye17, YinYan23}.
\end{enumerate}
The works cited above address a broad range of topics, such as the physical derivations of the models, various notions of well-posedness, existence of invariant measures, asymptotic behaviours, and large deviation principles. 

In what follows, we restrict our attention to the analytical studies concerning the existence of solutions. In particular, for many of the problems listed above, the existence of \emph{analytically weak martingale} or \emph{analytically weak pathwise} solutions has been studied in separate works, which we now briefly summarise:
\begin{enumerate}
\item A class of stochastic damped Navier--Stokes equations: Global martingale weak and pathwise weak solutions are established in~\cite{KinMoh23, KinCipMoh25} for $d\in \{2,3\}$ under Dirichlet boundary condition. Global pathwise strong solutions are studied in~\cite{GaoLiu19b,LiuGao18}, however there appears to be a gap in the argument with Dirichlet boundary condition; see Section~\ref{sec:scbf}.
\item A class of stochastic generalised Burgers equations: The existence of global pathwise strong solutions for the standard stochastic Burgers equations (with additive or multiplicative noise) is shown in~\cite{BrzGolNek14, DaGat95, DonGuoWuZho23, DonWuZho24}. The generalised case is discussed in~\cite{Kim06}, but only with additive noise.
\item Stochastic MHD equations: For $d \in\{2,3\}$, the existence of global martingale weak solutions is established in~\cite{SriSun99, Yam16} for bounded domains, and in~\cite{Mot22} for more general Poincar\'e domains. Pathwise uniqueness is also shown for $d=2$, which implies the existence of pathwise weak solutions.
\item Stochastic Boussinesq model for B\'enard convection: Local pathwise weak solutions exist in bounded domains for $d\in \{2,3\}$, as shown in~\cite{ChuMil10}, and these solutions are global for $d=2$.
\item Stochastic convective dynamo system: To the best of our knowledge, the existence of solutions for this system has not been established yet.
\item Stochastic thermo–magneto–micropolar fluid: Global martingale weak solutions exist for $d\in \{2,3\}$, as proven in~\cite{Yam14, Yam16m}.
\item Stochastic diffusive tropical climate model: Local pathwise weak solutions exist for two-dimensional unbounded domains, as shown in~\cite{YinYan23}.
\end{enumerate}
In contrast, the existence of \emph{analytically strong pathwise} solutions for these problems in bounded domains with physically relevant boundary conditions is currently unknown. 

There have been several attempts to establish abstract existence results for classes of stochastic magneto-hydrodynamical equations. We mention only those most relevant to the problems and techniques considered here. In~\cite{ChuMil10}, the existence of weak pathwise solutions to \eqref{equ:equation} is established in the case $\mathcal{R}(\Phi)=0$. In~\cite{DebGlaTem11}, the existence and uniqueness of local martingale and local pathwise solutions are obtained under assumptions which, in particular, treat $\mathcal{R}$ essentially as a Lipschitz perturbation. More recently, maximal strong solutions were studied in~\cite{GooCriLan24} under different structural and growth assumptions from ours, but including gradient-type noise. In a different direction, \cite{BesHauRaz15} treats a class of stochastic hydrodynamical systems driven by multiplicative jump noise, but under structural assumptions tailored primarily to standard convective-type nonlinearities.

The purpose of the present work is to formulate a flexible set of structural and growth assumptions under which analytically strong pathwise solutions can be obtained for a broad class of stochastic magneto-hydrodynamic-type systems with multiplicative noise. While the underlying strategy is based on established compactness, localisation, and pathwise uniqueness arguments, the assumptions used here are tailored to accommodate nonlinearities that fall outside some of the existing formulations. This requires
several refinements of the standard estimates, particularly in the treatment
of non-Lipschitz lower-order terms and coupled nonlinear structures. Under these assumptions, we establish local-in-time strong pathwise solutions for $d\in\{1,2,3\}$ and global-in-time solutions for $d\in\{1,2\}$ under additional growth conditions. A central contribution is the systematic verification of the abstract hypotheses for a range of physically motivated systems, thereby clarifying when strong solution theories are available for models that have previously been treated only in weaker senses, or not treated rigorously in the stochastic setting.

The proofs are based on Galerkin approximations combined with compactness arguments, but the analysis requires substantially more delicate estimates than in previous works. In particular, inspired by~\cite{GlaZia09, GlaTem11}, we derive strong moment bounds and establish a Cauchy property for the approximate solutions up to a suitable stopping time, despite the absence of any inherent cancellation structure in the equations. The estimates employ a Gronwall-type lemma for stochastic processes, and a maximal solution is
obtained by a localisation and continuation argument. Under an additional growth assumption, the strong energy estimates exclude finite-time blow-up and yield global-in-time existence for $d\in \{1,2\}$.

This paper is organised as follows:
\begin{itemize}
    \item Section~\ref{sec:prelim} sets up the functional framework in which the problem is posed and summarises the standing assumptions used throughout the paper.
    \item Section~\ref{sec:examples} presents several nonlinear physical models that satisfy the assumptions outlined in Section~\ref{sec:prelim}.
    \item Section~\ref{sec:loc sol} establishes the existence and uniqueness of a local strong pathwise solution to problem~\eqref{equ:equation}.
    \item Section~\ref{sec:max sol} extends the local strong solution to a maximal (random) time interval for~$d\in\{1,2,3\}$, and shows that this maximal strong solution is global when $d\in \{1,2\}$.
    \item Auxiliary results used throughout the paper are collected in
    Appendix~\ref{sec:app}. In particular, we state there the abstract comparison
    and Cauchy criterion, together with the stochastic Gronwall lemma used in the
    analysis.
\end{itemize}
Our main results are the existence of a unique maximal strong pathwise solution for $d\in\{1,2,3\}$ (Theorem~\ref{the:max strong sol}) and the existence of a global strong pathwise solution for $d\in\{1,2\}$ (Theorem~\ref{the:global strong 2d}).

\section{Preliminaries}\label{sec:prelim}

\subsection{Notations}\label{sec:notation}

We begin by defining some notations used in this paper. For $p\in [1,\infty]$ and $k\in \bb{N}$, the function spaces $L^p$ and $W^{k,p}$ denote, respectively, the Lebesgue space of $p$-th integrable functions and the Sobolev space of $L^p$ functions whose weak derivatives up to order $k$ also belong to $L^p$, respectively. As usual, we set $H^k := W^{k,2}$.

{If $X$ is a Banach space, $L^p(0,T; X)$ and $W^{k,p}(0,T;X)$ denote, respectively, the usual Lebesgue and Sobolev spaces of strongly measurable functions on $(0,T)$ taking values in $X$. Similarly, the space $L^p(\Omega; X)$ denotes the space of $X$-valued strongly measurable random variables with finite $p$-th moment, where $(\Omega,\mathcal{F},\bb{P})$ is a probability space. The space $C([0,T];X)$ and $C_\mathrm{w}([0,T]; X)$ denote the spaces of functions on $[0,T]$ taking values in $X$ that are continuous with respect to the strong and weak topologies of $X$, respectively.}

In the analysis, the constant $C$ denotes a
generic constant which may take different values at different occurrences. If the dependence of $C$ on a variable, e.g.~$s$, is highlighted, we will write $C_s$.

\subsection{Functional setting}
\label{subsec:fun set}

Let $\mathscr{O}\subset \bb{R}^d$, for $d\in \{1,2,3\}$, denote an open bounded domain with $C^2$-smooth boundary or a convex Lipschitz domain.
Let $(\mathcal{H}, \abs{\cdot})$ be a separable Hilbert space {and let} $\mathcal{A}$ be an unbounded, self-adjoint, positive linear operator on $\mathcal{H}$, defined on a domain $\mathrm{D}(\mathcal{A})\subset \mathcal{H}$, such that $\mathcal{A}$ is a bijection onto $\mathcal{H}$ with a compact inverse. Consequently, there exists an orthonormal basis $\{e_k\}_{k\in\bb{N}}$ for $\mathcal{H}$ consisting of eigenfunctions of $\mathcal{A}$. The associated positive eigenvalues $\{\lambda_k\}_{k\in\bb{N}}$ form an unbounded increasing sequence. Given $\alpha>0$, let
\begin{align*}
	\mathrm{D}(\mathcal{A}^\alpha) = \left\{v\in \mathcal{H} : \sum_{k=1}^\infty \lambda_k^{2\alpha} \abs{\inpro{v}{e_k}_\mathcal{H}}^2 < \infty\right\}.
\end{align*} 
For $v\in \mathrm{D}(\mathcal{A}^\alpha)$, we define
\[
\mathcal{A}^\alpha v:= \sum_{k=1}^\infty \lambda_k^{\alpha} \inpro{v}{e_k} e_k
\quad\text{and}\quad
\abs{v}_{\mathrm{D}(\mathcal{A}^\alpha)}^2:= \abs{\mathcal{A}^{\alpha} v}^2 = \sum_{k=1}^\infty \lambda_k^{2\alpha} \abs{\inpro{v}{e_k}_\mathcal{H}}^2.
\]

Let $\mathcal{V}=\mathrm{D}(\mathcal{A}^{\frac12})$, equipped with the norm $\norm{v}:= \abs{v}_{\mathrm{D}(\mathcal{A}^\frac12)}$ for any $v\in\mathcal{V}$. Let $\mathcal{V}'$ denote the dual of $\mathcal{V}$ with respect to the inner product $\inpro{\cdot}{\cdot}_\mathcal{H}$ of $\mathcal{H}$. Thus, we have a quartet of compactly embedded Hilbert spaces $\mathrm{D}(\mathcal{A}) \hookrightarrow \mathcal{V}\hookrightarrow \mathcal{H} \hookrightarrow \mathcal{V}'$. Let $\inpro{u}{v}$ denote the duality pairing between $u\in \mathcal{V}$ and $v\in \mathcal{V}'$, such that $\inpro{u}{v}=\inpro{u}{v}_{\mathcal{H}}$ whenever $u\in\mathcal{V}$ and $v\in \mathcal{H}$.

\subsection{Assumptions}\label{sec:assum}

We state some structural assumptions satisfied by the bilinear form $\mathcal{B}$. A typical example that we have in mind is 
\[
\mathcal{B}(v_1,v_2)= \bb{P}\big[(v_1\cdot\nabla)v_2\big],
\]
where $\bb{P}$ is the Helmholtz--Leray projector onto the space of divergence-free vector fields.
\\[1ex]
\textbf{Conditions (B)}:
\begin{itemize}
	\item[(B1)] $\mathcal{B}:\mathcal{V}\times \mathcal{V} \to \mathcal{V}'$ is a bilinear continuous map.
	\item[(B2)] For any $v_1,v_2,v_3 \in \mathcal{V}$,
	\begin{equation}\label{equ:B antisym}
		\inpro{\mathcal{B}(v_1,v_2)}{v_3} = -\inpro{\mathcal{B}(v_1,v_3)}{v_2}.
	\end{equation}
	\item[(B3)] There exists $C>0$ such that for any $v_1,v_2,v_3\in \mathcal{V}$,
	\begin{equation}\label{equ:est B 3d alt}
		\abs{\inpro{\mathcal{B}(v_1,v_2)}{v_3}}
		\leq
		C \norm{v_1} \norm{v_2} \abs{v_3}^{\frac12} \norm{v_3}^{\frac12}.
	\end{equation}
	\item[(B4)] There exists $C>0$ such that for any $v_1\in \mathcal{V}$ and $v_2\in \mathrm{D}(\mathcal{A})$,
	\begin{equation}\label{equ:est B with A}
		\abs{\mathcal{B}(v_1,v_2)}
		\leq
		\begin{cases}
			C \abs{v_1}^{\frac12} \norm{v_1}^{\frac12} \norm{v_2}^{\frac12} \abs{\mathcal{A}v_2}^{\frac12}, &\text{if $d=2$},
			\\[1ex]
			C \norm{v_1} \norm{v_2}^{\frac12} \abs{\mathcal{A} v_2}^{\frac12}, &\text{if $d\le 3$}.
		\end{cases}
	\end{equation}
	\item[(B5)] There exists $\alpha\in (0,\frac14)$, $p>0$, and $C>0$ such that, for every $v\in \mathrm{D}(\mathcal{A})$, one has $\mathcal{B}(v,v)\in \mathrm{D}(\mathcal{A}^\alpha)$ with the estimate:
	\begin{align}
		\label{equ:V norm B}
		\abs{\mathcal{B}(v,v)}_{\mathrm{D}(\mathcal{A}^{\alpha})}
		&\leq
		C \left(1+ \norm{v}^p\right) \left(1+\abs{\mathcal{A} v}\right).
	\end{align}
\end{itemize}
In particular, the antisymmetry property \eqref{equ:B antisym} implies $\inpro{\mathcal{B}(v_1,v_2)}{v_2}=0$ for any $v_1,v_2\in \mathcal{V}$. Subsequently, we will write $\mathcal{B}(v):= \mathcal{B}(v,v)$ for simplicity.
\\

We now state the assumptions satisfied by the nonlinear operator $\mathcal{R}$. A representative example to keep in mind when $v$ is a vector field is
\[
\mathcal{R}(v)=\alpha v+ \beta \bb{P}\big[ |v|^{r-1} v\big] + \bb{P}\big[(\phi\cdot\nabla)v\big], 
\]
where $\bb{P}$ is the Helmholtz--Leray projector, $\phi$ is a given smooth vector field, $\alpha\in \bb{R}$, $\beta\geq 0$, and $r\geq 1$. Another typical example that we consider when $v$ is a one-dimensional scalar field is
\[
\mathcal{R}(v)= \alpha v + \beta |v|^{r-1} v + \gamma \partial_x (\abs{v}^p),
\]
for some $\alpha\in \bb{R}$, $\beta,\gamma\geq 0$, $r \geq 1$, and $p\geq 2$.
\\[1ex]
\textbf{Conditions (R)}:
\begin{itemize}
	\item[(R1)] {$\mathcal{R}:\mathcal{V}\to \mathcal{H}$ is a continuous map defined by $\mathcal{R}=\mathcal{S}+\mathcal{F}$ where $\mathcal{S}, \mathcal{F}:\mathcal{V}\to\mathcal{H}$ satisfy, for any $v\in\mathcal{V}$,}
	\begin{equation}\label{equ:Rv v pos}
		\inpro{\mathcal{S}(v)}{v} \geq 0 \;\text{ and }\; \abs{\mathcal{F}(v)} \leq C\left(1+\norm{v}\right).
	\end{equation}
	\item[(R2)] There exists $C>0$ and $r\geq 1$ such that, for any $v\in \mathcal{V}$,
    \begin{align}\label{equ:est R v in v}
		\abs{\mathcal{R}(v)}
		&\leq
		C \left(1+ \norm{v}^{r-1} \right) \norm{v}.
    \end{align}
	\item[(R3)] There exists $C>0$ and $r\geq 1$ such that, for any {$v_1,v_2\in\mathcal{V}$}
	\begin{equation}\label{equ:est Rv1 minus v2}
        \abs{\mathcal{R}(v_1)-\mathcal{R}(v_2)} \le C\left(1+\norm{v_1}^{r-1}+\norm{v_2}^{r-1}\right) \norm{v_1-v_2}
	\end{equation}
	\item[(R4)] There exists $\beta\in (0,\frac14)$, $q>0$, and $C>0$ such that, for every $v\in \mathrm{D}(\mathcal{A})$, one has $\mathcal{R}(v)\in \mathrm{D}(\mathcal{A}^{\beta})$ with the estimate:
	\begin{align}\label{equ:est R vv}
        \abs{\mathcal{R}(v)}_{\mathrm{D}(\mathcal{A}^\beta)}
		&\leq
		C \left(1+ \norm{v}^q \right) \left(1+\abs{\mathcal{A}v}\right).
	\end{align}
\end{itemize}

\vspace{1ex}
The stochastic term in \eqref{equ:equation} can be written formally as
\begin{equation}\label{equ:g form}
	g(\Phi) \dot{W}= \sum_{k=1}^\infty g_k(\Phi) \dot{\beta}_k,
\end{equation}
where $\beta_k$ are independent standard Brownian motions. We need to impose some assumptions on $g= \{g_k\}_{k\in \bb{N}}$. To this end, we first define the following notion: for any normed spaces $Y$ and $Z$, a progressively measurable map $h:\Omega\times [0,\infty)\times Y\to Z$ is said to be uniformly Lipschitz with constant $L_Y$ if, for all $u,v\in Y$ and $(\omega,t)\in \Omega\times[0,\infty)$,
\begin{align*}
	\abs{h(\omega,t,u)-h(\omega,t,v)}_{Z} \leq L_Y \abs{u-v}_Y
\quad\text{and}\quad
	\abs{h(\omega,t,u)}_{Z} \leq L_Y \left(1+\abs{u}_Y\right).
\end{align*}
The collection of all such mappings is denoted by $\text{Lip}(Y, Z)$.

For our analysis, we shall assume the following condition on $g$. 
\\[1ex]
\textbf{Condition (G)}:
A progressively measurable map $g:\Omega\times [0,\infty)\times \mathcal{H} \to \ell^2(\mathcal{H})$ is given such that
\begin{align}\label{equ:cond G}
	g\in \text{Lip}\big(\mathcal{H}, \ell^2(\mathcal{H})\big) \cap \text{Lip}\big(\mathcal{V}, \ell^2(\mathcal{V})\big) \cap \text{Lip}\big(\mathrm{D}(\mathcal{A}), \ell^2(\mathrm{D}(\mathcal{A}))\big).
\end{align}

Finally, to obtain a global strong pathwise solution when $d\leq 2$, we impose, in addition to the above assumptions, the following growth condition: there exist constants $C>0$ and $p\geq 0$ such that, for any $v\in\mathcal{V}$,
\begin{align}\label{equ:ass global}
	\abs{\mathcal{R}(v)} \leq C\left(1+ \abs{v}^p \abs{\mathcal{A} v}^{\frac12}\right) \norm{v}.
\end{align}
Assumption~\eqref{equ:ass global} will only be used in Section~\ref{sec:strong}.

\section{Physical models}\label{sec:examples}

We now present several physically relevant models that fall within the scope of our framework. We note, however, that our assumptions are not limited to these thermo-magneto-fluid models and can accommodate a broader class of nonlinear stochastic PDEs arising in various contexts.

\subsection{A class of stochastic damped Navier--Stokes equations}\label{sec:scbf}

Let $\mathscr{O}\subset\bb{R}^d$, where $d\in \{2,3\}$. The stochastic damped Navier--Stokes (or Brinkman--Forchheimer) equations~\cite{BrzFer19, CaiJiu08, GaoLiu19b, GaoLiu19, HajRob17, KinCipMoh25, KinMoh23, LiuGao18} with Dirichlet (no-slip) boundary conditions read:
\begin{subequations}\label{equ:scbf}
	\begin{alignat}{2}
		&\partial_t u - \nu \Delta u + (u\cdot\nabla)u + P(u) + \nabla p = g(u)\dot{W}
		\; && \quad\text{in $\mathscr{O}$,}
		\label{equ:scbf a}
		\\[1ex]
		&\divg u = 0
		\; && \quad\text{in $\mathscr{O}$,}
		\label{equ:scbf b}
		\\[1ex]
		&u(0,x)= u_0(x) 
		\; && \quad\text{in $\mathscr{O}$,}
		\label{equ:scbf c}
		\\[1ex]
		&u = 0
		\; && \quad\text{on $\partial\mathscr{O}$,}
		\label{equ:scbf d}
	\end{alignat}
\end{subequations}
where $P(u)=\alpha u +\beta \abs{u}^{r-1} u$ is a polynomial damping term, and $g$ satisfies~\eqref{equ:cond G}. We fix 
\begin{equation}\label{equ:r}
	r\in
	\begin{cases}
		[1,\infty) & \text{if } d=2 \\
		[1,3] & \text{if } d=3.
	\end{cases}
\end{equation}
In~\eqref{equ:scbf}, $u:\mathscr{O}\to \bb{R}^d$ is the fluid velocity and $p:\mathscr{O}\to \bb{R}$ is the pressure. The positive coefficients $\nu$, $\alpha$, and $\beta$ respectively denote the kinematic viscosity (Brinkman coefficient), the Darcy coefficient (permeability of porous medium), and the Forchheimer coefficient. The system~\eqref{equ:scbf} describes the flow of a viscous fluid at sufficiently high velocity and moderate porosity, perturbed by a stochastic forcing term $g(u)\dot{W}$. Our results also extend to the case of Navier-type (or free) boundary conditions~\cite{AmrRej14, Zia98}.

The existence and uniqueness of probabilistically and analytically strong
solutions to \eqref{equ:scbf} on the torus, or equivalently under periodic
boundary conditions, were essentially established in~\cite{GaoLiu19b,LiuGao18}.
On more general Poincar\'e domains, related well-posedness results were
obtained in~\cite{KinCipMoh25, KinMoh23}; however, these solutions are
analytically weak, even when they are strong in the probabilistic sense. 

Obtaining analytically strong solutions under Dirichlet boundary conditions is substantially more challenging. A key difficulty is that the identity
\begin{align*}
	\inpro{\mathcal{A}u}{|u|^{r-1} u}
	&=
	\int_\mathscr{O} \abs{\nabla u(x)}^2 \abs{u(x)}^{r-1} \dx
	+
	\frac{4(r-1)}{(r+1)^2}
	\int_\mathscr{O}
	\abs{\nabla \abs{u(x)}^{\frac{r+1}{2}}}^2 \dx
\end{align*}
is available in the periodic setting, but does not directly extend to bounded domains with Dirichlet boundary condition in the projected formulation. Indeed, although $|u|^{r-1}u$ vanishes on $\partial\mathscr O$, it is generally not divergence-free, while $\bb{P}(|u|^{r-1}u)$ need not vanish on the boundary. Moreover, the Helmholtz projection $\bb{P}$ and the Dirichlet Laplacian do not commute. Related observations are made in~\cite{KinMoh23,LiKim24}, implying that the periodic-domain arguments used in~\cite{GaoLiu19,LiuGao18} cannot be transferred directly to the Dirichlet setting. 
In this work, we address this difficulty and establish the existence and uniqueness of pathwise strong solutions for $r$ satisfying \eqref{equ:r}, including the case of Dirichlet boundary condition. When $d=2$ and $r\in[1,3]$, the solution exists globally in time.

Define the spaces:
\begin{align*}
    \mathcal{V} &:= \{u\in H^1_0(\mathscr{O}): \divg u=0\},
    \\
    \mathcal{H} &:= \{u\in L^2(\mathscr{O}) : \divg u=0 \text{ and } \left. u\cdot n \right|_{\partial \mathscr{O}}= 0\}.
\end{align*}
In the definition of $\mathcal{H}$ here and in what follows, the condition $\divg u=0$ is understood in the sense of distributions, while the trace of the normal component of $u$ on the boundary is interpreted in a suitable generalised sense (see~\cite[Chapter I, Section~1.3]{Tem01} and~\cite[Section~III.2]{Gal11}).
By means of the Helmholtz--Leray projector $\bb{P}:L^2(\mathscr{O}) \to \mathcal{H}$, noting that $\bb{P}v=v$ and $\bb{P}[g(v)\dot{W}]= \sum_{k=1}^\infty g_k(v)\dot{\beta}_k$ due to~\eqref{equ:g form} and \eqref{equ:cond G}, we can project \eqref{equ:scbf} onto the space of divergence-free vector fields to formulate problem \eqref{equ:scbf} in the form \eqref{equ:equation}, with
\begin{alignat}{2}
    &\mathcal{A}v =
    -\nu\bb{P}\Delta v,
    &\quad &\forall v\in\mathrm{D}(\mathcal{A}),
    \label{equ:A v}
    \\
    \label{equ:B v1 v2}
    &\mathcal{B}(v_1,v_2) = \bb{P}\big[(v_1\cdot\nabla)v_2\big],&\quad &\forall v_1,v_2 \in \mathcal{V},
    \\
    \label{equ:Rv scbf}
    &\mathcal{R}(v)
    = \alpha v+ \beta \bb{P} \big[ |v|^{r-1} v \big], &\quad &\forall v\in \mathcal{V},
\end{alignat}
where $\mathrm{D}(\mathcal{A}) := H^2(\mathscr{O})\cap \mathcal{V}$.
The operator $\mathcal{A}:\mathrm{D}(\mathcal{A})\to\mathcal{H}$ is called the \emph{Dirichlet--Stokes operator}.

The map $\mathcal{B}$ satisfies \eqref{equ:B antisym} by the definition of $\bb{P}$ and simple calculations. It also satisfies inequalities \eqref{equ:est B 3d alt} and \eqref{equ:est B with A} by a standard argument using H\"older's inequality and Sobolev embedding (see~\cite{Tem01}). 
Furthermore, by the boundedness of $\bb{P}$ and H\"older's inequality (with exponents $p=q=4$ when $d=2$, and $p=6,q=3$ when $d=3$), we have
\begin{align}\label{equ:Bv L2}
    \abs{\mathcal{B}(v,v)}^2 \leq \abs{v}_{L^p}^2 \abs{\nabla v}_{L^q}^2 \leq C \big(\abs{v} \norm{v}\big) \big(\norm{v} \abs{\mathcal{A}v} \big) \leq
    C\norm{v}^3 \abs{\mathcal{A} v},
\end{align}
where in the penultimate step we used the Gagliardo--Nirenberg inequalities.
To show \eqref{equ:V norm B}, we first note that $(v\cdot\nabla)v\in H^1(\mathscr{O})$ for any $v\in \mathrm{D}(\mathcal{A})$, since by the H\"older and the Gagliardo--Nirenberg inequalities,
\begin{align}\label{equ:v grad v H1}
    \abs{(v\cdot\nabla)v}_{H^1}^2
    &\leq
    \abs{v}_{L^4}^2 \abs{\nabla v}_{L^4}^2
    +
    \abs{v}_{L^\infty}^2 \abs{\mathcal{A}v}^2 + \abs{\nabla v}_{L^p}^2 \abs{\nabla v}_{L^q}^2 
    \leq
    C \norm{v} \abs{\mathcal{A} v}^3.
\end{align}
In particular, this implies $\mathcal{B}(v,v)\in \mathrm{D}(\mathcal{A}^{\frac18})$ by the boundedness of $\bb{P}$ in $H^{\frac14}(\mathscr{O})$ and the characterisation of the domains of fractional powers of the Dirichlet--Stokes operator (see~\cite{FefHajRob22, KunWei17}).
Furthermore, by interpolation, we infer that
\begin{align}\label{equ:Bvv A18}
    \abs{\mathcal{B}(v,v)}_{\mathrm{D}(\mathcal{A}^{\frac18})}^2 
    &\leq
    \abs{(v\cdot\nabla)v}_{H^{\frac12}}^2 
    \leq 
    \abs{(v\cdot\nabla)v}_{L^2}
    \abs{(v\cdot\nabla)v}_{H^1}
    \leq
    C \norm{v}^2 \abs{\mathcal{A} v}^2,
\end{align}
where in the last step we used \eqref{equ:Bv L2} and \eqref{equ:v grad v H1}. This proves \eqref{equ:V norm B}. Thus, condition (B) in Section~\ref{sec:assum} is satisfied by the map $\mathcal{B}$.

The map $\mathcal{R}$ clearly satisfies \eqref{equ:Rv v pos} with $\mathcal{F}(v)=0$. It remains to show that the map satisfies \eqref{equ:est R v in v}, \eqref{equ:est Rv1 minus v2}, and \eqref{equ:est R vv} for the range of $r$ in \eqref{equ:r}.
First, for any $v\in\mathcal{V}$, we have by H\"older's inequality and the Sobolev embedding $\mathcal{V} \hookrightarrow L^{2r}$, we have
\begin{align}\label{equ:Rv L2}
	\abs{\mathcal{R}(v)} \leq C\left(\abs{v} + \abs{v}_{L^{2r}}^{r}\right) 
	\leq C \left(1+ \norm{v}^{r-1}\right) \norm{v},
\end{align}
thus showing \eqref{equ:est R v in v}.
By a similar argument, for any $v_1,v_2\in \mathcal{V}$ and $v_3\in \mathcal{H}$, we obtain
\begin{align*}
    \abs{\mathcal{R}(v_1)-\mathcal{R}(v_2)}
    &\leq
    C\left(1+\abs{v_1}_{L^{2r}}^{r-1} + \abs{v_2}_{L^{2r}}^{r-1} \right) \abs{v_1-v_2}_{L^{2r}}
    \\
    &\leq
    C\left(1+\norm{v_1}^{r-1} +\norm{v_2}^{r-1} \right) \norm{v_1-v_2},
\end{align*}
by the Sobolev embedding $\mathcal{V}\hookrightarrow L^{2r}$ and Young's inequality. This shows~\eqref{equ:est Rv1 minus v2}. Finally, we will show~\eqref{equ:est R vv}. To this end, note that $\abs{v}^{r-1} v\in H^1(\mathscr{O})$ for any $v\in \mathrm{D}(\mathcal{A})$ since by H\"older's inequality and the Sobolev embedding $H^1\hookrightarrow L^{2r}$,
\begin{align}\label{equ:Rv H1}
	\abs{\abs{v}^{r-1} v}_{H^1} \leq 
	\abs{v}_{L^{2r}}^r + \abs{v}_{L^{2r}}^{r-1} \abs{\nabla v}_{L^{2r}}
	\leq
	C\norm{v}^{r-1} \abs{\mathcal{A} v}.
\end{align}
Hence, by the boundedness of $\bb{P}$ in $\mathrm{D}(\mathcal{A}^{\frac18})$ and interpolation, noting \eqref{equ:Rv L2} and \eqref{equ:Rv H1}, we obtain
\begin{align}\label{equ:Rv A18}
	\abs{\mathcal{R}(v)}_{\mathrm{D}(\mathcal{A}^{\frac18})}^2
	\leq
	C\abs{v}_{H^{\frac12}}^2 + C\abs{ \abs{v}^{r-1} v}_{H^{\frac12}}^2
	&\leq
	C\abs{v} \norm{v} + C \abs{ \abs{v}^{r-1} v}_{L^2} \abs{ \abs{v}^{r-1} v}_{H^1}
	\nonumber \\
	&\leq
	C\left(1+\norm{v}^{2(r-1)} \right) \left(1+\abs{\mathcal{A} v}^2 \right),
\end{align}
proving \eqref{equ:est R vv}. This shows condition (R) in Section~\ref{sec:assum} is satisfied by the map $\mathcal{R}$.

Finally, to show that the solution is global when $d=2$ and $r\in [1,3]$, we need to verify \eqref{equ:ass global}. For ease of presentation, assume $r=3$, noting that a similar argument also holds when $r\in [1,3)$. By the H\"older and Gagliardo--Nirenberg inequalities in two-dimensional domains,
\begin{align}\label{equ:Rv global}
	\abs{\mathcal{R}(v)}
	\leq
	C\abs{v} + C\abs{\abs{v}^2 v}
	&\leq
	C\abs{v} + C\abs{v}_{L^4}^2 \abs{v}_{L^\infty} 
	\nonumber\\
	&\leq
	C\abs{v} + C\abs{v}_{L^2}^{\frac32} \norm{v} \abs{\mathcal{A}v}^{\frac12}.
\end{align}
This implies \eqref{equ:ass global}, as required.

\begin{remark}
In certain physical applications, it is argued that Navier-type (Hodge or perfect slip) boundary conditions specified by
\begin{align}\label{equ:navier bc}
    u\cdot n=0 \quad \text{and} \quad \mathrm{curl}\, u\times n=0 \quad \text{on $\partial \mathscr{O}$}
\end{align}
are more appropriate~\cite{AceAmrCon19}. In this case, the operator $\mathcal{A}=-\nu \bb{P}\Delta$ is the \emph{Hodge--Stokes operator}, and the corresponding function spaces are
\begin{align*}
    \mathcal{V} &:= \{u\in H^1(\mathscr{O}): \divg u=0 \text{ in $\mathscr{O}$, } \left. u\cdot n \right|_{\partial \mathscr{O}}= 0 \,\text{ and } \left.\mathrm{curl}\, u\times n\right|_{\partial\mathscr{O}}=0 \},
    \\
    \mathcal{H} &:= \{u\in L^2(\mathscr{O}): \divg u=0 \text{ in $\mathscr{O}$ and } \left. u\cdot n \right|_{\partial \mathscr{O}}= 0\},
    \\
    \mathrm{D}(\mathcal{A}) &:= H^2(\mathscr{O}) \cap \mathcal{V}.
\end{align*}
The verification of condition (B) and condition (R) proceeds in the same manner. To verify \eqref{equ:Bvv A18} and \eqref{equ:Rv A18}, the corresponding results in~\cite{Kun25} for the Hodge--Stokes operator are used.
\end{remark}

\subsection{A class of stochastic generalised Burgers equations}\label{sec:burgers}

Generalised Burgers equations constitute a fundamental class of nonlinear
convection--diffusion models. They arise as simplified models for fluid
motion and turbulence, nonlinear wave propagation, traffic flow, and
transport phenomena in porous media; see~\cite{BecKha07} and the references
therein.

Let $\mathscr O$ be a one-dimensional domain ($d=1$), and let
$F\in C^1(\mathbb R;\mathbb R)$ be a prescribed flux. We consider stochastic
generalised Burgers equations of the form
\begin{subequations}\label{equ:sb}
	\begin{alignat}{2}
		&\partial_t u - \nu \Delta u + \partial_x\big(F(u)\big) + P(u)
		= g(u)\dot W
		\; && \quad\text{in $\mathscr O$,}
		\label{equ:sb a}
		\\[1ex]
		&u(0,x)= u_0(x)
		\; && \quad\text{in $\mathscr O$,}
		\label{equ:sb b}
		\\[1ex]
		&u = 0
		\; && \quad\text{on $\partial\mathscr O$.}
		\label{equ:sb c}
	\end{alignat}
\end{subequations}
Here, $u$ is a scalar-valued function, and the homogeneous Dirichlet condition \eqref{equ:sb c} may also be replaced by a periodic boundary.

Our framework applies, for instance, to the polynomial-type damping $P(u)=\alpha u +\beta |u|^{r-1} u$ with $\alpha,\beta\geq 0$ and $r\in [1,\infty)$, and polynomial-type flux $F(u)=\frac1p |u|^p$ with $p\in [2,\infty)$. Hence,
\[
\partial_x \big(F(u)\big)=|u|^{p-2}u \,\partial_x u.
\]
When $p=2$ and $\alpha=\beta=0$, equation~\eqref{equ:sb} reduces to the standard stochastic
Burgers equation. In this standard case, the existence of strong pathwise
solutions has been studied in~\cite{DaGat95}. Related higher-dimensional
and vector-valued variants have also been considered in
\cite{BrzGolNek14, DonGuoWuZho23, DonWuZho24}. 

For non-quadratic fluxes, corresponding equations with additive noise have
been analysed in~\cite{Bor13a, Bor13b, Kim06}, where existence and finer bounds in
various Sobolev norms are obtained. The present framework gives existence
of strong pathwise solutions for one-dimensional generalised Burgers
equations with multiplicative noise. In particular, our results apply to
stochastic counterparts of the deterministic generalised Burgers equations
considered in~\cite{RaoSacRam03}. The strong solution is global when $p\in [2,3]$ and $r\in [1,3]$.

We describe the functional setting next.
The relevant function spaces are $\mathcal{V}:= H^1_0(\mathscr{O})$ and $\mathcal{H}=L^2(\mathscr{O})$. The relevant linear operator is $\mathcal{A}=-\nu\Delta$, where $\Delta$ is the Dirichlet Laplacian. The nonlinear operators are defined by \begin{align*}
	\mathcal{B}(u)&=\mathcal{F}(u)=0,
	\\
	\mathcal{S}(u)&= \partial_x\big(F(u)\big) + P(u),
\end{align*}
where $F(u)=\frac1p |u|^p$ and $P(u)=\alpha u +\beta |u|^{r-1} u$. 

It remains to verify that condition (R) in our assumptions is satisfied. First, note that with Dirichlet or periodic boundary condition, 
\begin{align*}
	\inpro{\mathcal{S}(u)}{u}= \inpro{P(u)}{u}\geq 0.
\end{align*}
Furthermore, by H\"older 's inequality and the Sobolev embedding $H^1\hookrightarrow L^\infty$ in one-dimensional domains,
\begin{align}\label{equ:Rv L2 burgers}
	\abs{\mathcal{R}(v)} 
	&\leq
	\alpha \abs{v} + \beta \abs{v}_{L^\infty}^{r-1} \abs{v} + \abs{v}_{L^\infty}^{p-1} \abs{\partial_x v}
	\leq 
	\alpha \norm{v} + \beta \norm{v}^r + \norm{v}^p,
\end{align}
verifying \eqref{equ:est R v in v}. Similarly, for any $v_1,v_2\in \mathcal{V}$, we infer that
\begin{align*}
	\abs{\mathcal{R}(v_1)-\mathcal{R}(v_2)} 
	&\leq 
	\alpha \abs{v_1-v_2} + \beta \abs{|v_1|^{r-1} v_1 - |v_2|^{r-1} v_2} 
	\\
	&\quad
	+
	\abs{|v_1|^{p-2} v_1 (\partial_x v_1-\partial_x v_2)} + 
	\abs{(|v_1|^{p-2} v_1 - |v_2|^{p-2} v_2) \partial_x v_2}
	\\
	&\leq
	\alpha \abs{v_1-v_2} + C\left(1+\abs{v_1}_{L^\infty}^{r-1}+\abs{v_2}_{L^\infty}^{r-1} \right) \abs{v_1-v_2}
	\\
	&\quad
	+
	C \abs{v_1}_{L^\infty}^{p-1} \norm{v_1-v_2} + C\left(1+\abs{v_1}_{L^\infty}^{p-2}+ \abs{v_2}_{L^\infty}^{p-2}\right) \abs{v_1-v_2}_{L^\infty} \norm{v_2}
	\\
	&\leq
	C\left(1+ \norm{v_1}^{\overline{p}-1}+ \norm{v_2}^{\overline{p}-1}\right) \norm{v_1-v_2},
\end{align*}
where $\overline{p}:=\max\{r,p\}$. This shows \eqref{equ:est Rv1 minus v2}. Before proving \eqref{equ:est R vv}, we note that by the algebra property of $H^1$ in one-dimensional domains,
\begin{align*}
	\abs{\mathcal{R}(v)}_{H^1}
	&\leq
	\norm{u}^{p-1} \abs{\mathcal{A}u} + \alpha \norm{u} + \beta \norm{u}^r.
\end{align*}
As such, by interpolation inequalities and \eqref{equ:Rv L2 burgers},
\begin{align*}
	\abs{\mathcal{R}(v)}_{\mathrm{D}(\mathcal{A}^{\frac18})}^2
	\leq
	\abs{\mathcal{R}(v)}_{H^{\frac12}}^2
	&\leq
	\abs{\mathcal{R}(v)}_{L^2} \abs{\mathcal{R}(v)}_{H^1}
	\\
	&\leq
	C\left(1+ \norm{v}^{\overline{p}}\right) \left(1+ \norm{u}^r+ \norm{u}^{p-1} \abs{\mathcal{A}u}\right)
	\\
	&\leq
	C\left(1+ \norm{v}^{2\overline{p}}\right) \left(1+ \abs{\mathcal{A}u}\right),
\end{align*}
showing \eqref{equ:est R vv}. This completes the verification of condition (R).

Furthermore, the pathwise strong solution is global when $p\in [2,3]$ and $r\in [1,3]$. Indeed, for ease of presentation, we verify \eqref{equ:ass global} only for $p=r=3$. In this case, by the Gagliardo--Nirenberg inequalities,
\begin{align*}
	\abs{\mathcal{R}(v)}
	&\leq
	C\norm{v}+ C\abs{|v|^2 v}+ C\abs{|v|^2 \partial_x v}
	\\
	&\leq
	C\norm{v}+ C\abs{v}_{L^\infty}^2 \abs{v} + C \abs{v}_{L^\infty}^2 \abs{\partial_x v}
	\\
	&\leq
	C\norm{v} + C\abs{v}^{\frac32} \norm{v} \abs{\mathcal{A} v}^{\frac12},
\end{align*}
thus proving \eqref{equ:ass global}.

\subsection{Stochastic magnetohydrodynamic equations}\label{sec:mhd}

The dynamics of turbulent motion of plasmas and liquid metals~\cite{Bis03} can be described by the stochastic magnetohydrodynamic (MHD) equations~\cite{Mot22, SerTem83}, which is a coupled system of stochastic Navier--Stokes and stochastic Maxwell-type equations. Such model is important in the study of the magnetic reconnection problem in astrophysics~\cite{PriFor00} or in the design of nuclear fusion device. 

Let $\mathscr{O}\subset\bb{R}^d$, where $d\in \{2,3\}$. The stochastic MHD system read:
\begin{subequations}\label{equ:sme}
	\begin{alignat}{2}
		&\partial_t u - \nu \Delta u + (u\cdot\nabla)u - (B\cdot\nabla) B + \nabla \left(p+\frac12 \abs{B}^2\right) = g_1(u,B)\dot{W}_1
		\; && \quad\text{in $\mathscr{O}$,}
		\label{equ:sme a1}
		\\[1ex]
        \label{equ:sme a2}
        &\partial_t B - \kappa \Delta B + (u\cdot\nabla)B - (B\cdot\nabla)u = g_2(u,B) \dot{W}_2
        \; && \quad\text{in $\mathscr{O}$,}
        \\[1ex]
		&\divg u = \divg B= 0
		\; && \quad\text{in $\mathscr{O}$,}
		\label{equ:sme b}
		\\[1ex]
		&u(0,x)= u_0(x), \; B(0,x)= B_0(x)
		\; && \quad\text{in $\mathscr{O}$,}
		\label{equ:sme c}
		\\[1ex]
		&u = 0,
		\; && \quad\text{on $\partial\mathscr{O}$,}
		\label{equ:sme d}
        \\[1ex]
        \label{equ:sme e}
        &B\cdot n=0, \; \mathrm{curl}\; B\times n=0
        \; && \quad \text{on $\partial\mathscr{O}$.}
	\end{alignat}
\end{subequations}
The fields $u:\mathscr{O}\to \bb{R}^d$ and $B:\mathscr{O}\to \bb{R}^d$ denote the fluid velocity and the magnetic flux density, respectively. Here, $g_1\dot{W}_1$ and $g_2\dot{W}_2$ are the stochastic forcing terms, where $W_1$ and $W_2$ are independent real-valued Brownian motions. One can also add the term $\alpha u+ \beta \abs{u}^{r-1} u$ to the left-hand side of \eqref{equ:sme a1} to model a damped flow or a flow in porous media as in~\cite{AngCamCau23, TitTra19}.

The existence of martingale weak (in the sense of PDEs) solutions to the stochastic MHD systems with multiplicative noise for $d\leq 3$ is shown in~\cite{SriSun99, Yam16} for a bounded domain, and in~\cite{Mot22} for a Poincar\'e domain. Pathwise uniqueness of the solution is also derived for $d\leq 2$, which implies the existence of pathwise solutions (weak in the sense of PDEs). The argument in this paper establishes the existence and uniqueness of probabilistically and analytically strong solutions to the stochastic MHD system (including its damped modification), thus extending known results. Our theory also extends to the case of perfect slip boundary conditions for the fluid~\eqref{equ:navier bc} or periodic boundary conditions.

The appropriate functional setting for this problem will now be described. Define the spaces:
\begin{align*}
    V_1 &:= \{u\in H^1_0(\mathscr{O}): \divg u=0\},
    \\
    H_1 &:= \{u\in L^2(\mathscr{O}): \divg u=0 \text{ in $\mathscr{O}$ and } \left. u\cdot n \right|_{\partial \mathscr{O}}= 0\},
    \\
    V_2 &:= \{B\in H^1(\mathscr{O}): \divg B=0 \text{ in $\mathscr{O}$, } \left. B\cdot n \right|_{\partial \mathscr{O}}= 0 \,\text{ and } \left.\mathrm{curl}\, B\times n\right|_{\partial\mathscr{O}}=0 \},
    \\
    H_2 &:= H_1= \{B\in L^2(\mathscr{O}): \divg u=0 \text{ in $\mathscr{O}$ and } \left. u\cdot n \right|_{\partial \mathscr{O}}= 0\},
    \\
    \mathcal{V} &:= V_1 \times V_2,
    \\
    \mathcal{H} &:= H_1 \times H_2.
\end{align*}
As before, $\bb{P}:L^2(\mathscr{O})\to H_1$ is the Helmholtz--Leray projector. Let $A_1= -\nu \bb{P}\Delta$ be the Dirichlet--Stokes operator and $A_2= -\kappa \Delta$ be the Hodge Laplacian operator. Let $\mathcal{A}=A_1\oplus A_2$ with domain $\mathrm{D}(\mathcal{A})= \mathcal{V} \cap \left[H^2(\mathscr{O})\right]^2$ acting on $\Phi=(u,B)$. The map $\mathcal{B}$ is defined by
\begin{align*}
    \mathcal{B}(\Phi_1,\Phi_2)= 
    \begin{pmatrix} \bb{P}\big[(u_1\cdot\nabla)u_2- (B_1\cdot\nabla)B_2\big] \\[0.5ex] (u_1\cdot\nabla)B_2-(B_1\cdot\nabla)u_2 \end{pmatrix}
\end{align*}
for $\Phi_i=(u_i,B_i)\in \mathcal{V}$, where $i=1,2$. The stochastic term is $g(\Phi)\dot{W}= (g_1\dot{W}_1,\, g_2\dot{W}_2)^\top$. In this case, $\mathcal{R}(\Phi)=0$.

The map $\mathcal{B}$ satisfies \eqref{equ:B antisym}, as well as inequalities \eqref{equ:est B 3d alt} and \eqref{equ:est B with A}; see~\cite{Mot22}. The verification of \eqref{equ:V norm B} proceeds in the same manner as in \eqref{equ:Bv L2} and \eqref{equ:Bvv A18}.

\subsection{Stochastic Boussinesq model for the B\'enard convection}\label{sec:sbc}

A model for convection in a fluid due to B\'enard consists of a coupled system of Navier--Stokes and heat equations, employing Boussinesq approximation (see~\cite{AgrBerSat15, FoiManTem87, GalPad90} and references therein). One can also use the Brinkman--Forchheimer equation instead for the fluid flow, as in~\cite{AmiLepOta25, SaySemTri21, TitTra24}. 

Let $\mathscr{O}\subset\bb{R}^d$, where $d\in \{2,3\}$. Let $e$ be a given vector in $\bb{R}^d$ and suppose that the non-homogeneous boundary temperature data is fixed. By employing a transformation procedure~\cite{His91}, one can convert the problem into a system with homogeneous Dirichlet data, which reads:
\begin{subequations}\label{equ:sbc}
	\begin{alignat}{2}
		&\partial_t u - \nu \Delta u + (u\cdot\nabla)u + \alpha u +\beta |u|^{r-1} u + \nabla p + \theta e = g_1(u,\theta)\dot{W}_1
		\; && \quad\text{in $\mathscr{O}$,}
		\label{equ:sbc a1}
		\\[1ex]
        \label{equ:sbc a2}
        &\partial_t \theta - \kappa \Delta \theta + (u\cdot\nabla)\theta + (u\cdot\nabla)\phi = g_2(u,\theta) \dot{W}_2
        \; && \quad\text{in $\mathscr{O}$,}
        \\[1ex]
		&\divg u = 0
		\; && \quad\text{in $\mathscr{O}$,}
		\label{equ:sbc b}
		\\[1ex]
		&u(0,x)= u_0(x) 
		\; && \quad\text{in $\mathscr{O}$,}
		\label{equ:sbc c}
		\\[1ex]
		&u = 0, \; \theta=0
		\; && \quad\text{on $\partial\mathscr{O}$.}
		\label{equ:sbc d}
	\end{alignat}
\end{subequations}
In the above system, $u:\mathscr{O}\to \bb{R}^d$ and $\theta:\mathscr{O}\to \bb{R}$ denote the fluid velocity and the temperature, respectively. The function $\phi\in H^2(\mathscr{O})$ is assumed to be given, and $r$ satisfies \eqref{equ:r}. The non-negative coefficients $\alpha$ and $\beta$ are, respectively, the Darcy and the Forchheimer coefficients. Perfect slip boundary conditions for $u$ as defined by \eqref{equ:navier bc} can also be used.

The existence of probabilistically strong, analytically weak solutions to the problem~\eqref{equ:sbc} is established in~\cite{ChuMil10}. The argument in this paper proves the existence and uniqueness of probabilistically and analytically strong solutions to this problem for $d\leq 3$ and $r$ satisfying \eqref{equ:r}. This solution is global when $d=2$ and $r\in [1,3]$.

We now describe the functional setting for this system. Define the spaces:
\begin{align*}
    V_1 &:= \{v\in H^1_0(\mathscr{O}): \divg v=0\},
    \\
    H_1 &:= \{v\in L^2(\mathscr{O}): \divg v=0 \;\text{ and } \left. v\cdot n \right|_{\partial \mathscr{O}}= 0\},
    \\
    V_2 &:= H^1_0(\mathscr{O}),
    \\
    H_2 &:= L^2(\mathscr{O}),
    \\
    \mathcal{V} &:= V_1 \times V_2,
    \\
    \mathcal{H} &:= H_1 \times H_2.
\end{align*}
Let $A_1= -\nu \bb{P}\Delta$ be the Dirichlet--Stokes operator and $A_2=-\kappa\Delta$ be the Dirichlet Laplacian operator. Let $\mathcal{A}= A_1 \oplus A_2$ with domain $\mathrm{D}(\mathcal{A})= \mathcal{V} \cap \left[H^2(\mathscr{O})\right]^2$ acting on $\Phi=(u,\theta)$. The map $\mathcal{B}$ is given by
\begin{align*}
    \mathcal{B}(\Phi_1,\Phi_2)= 
    \begin{pmatrix} \bb{P}\big[(u_1\cdot\nabla)u_2\big] \\[0.5ex] (u_1\cdot\nabla)\theta_2 \end{pmatrix}
\end{align*}
for $\Phi_i=(u_i,\theta_i)\in \mathcal{V}$, where $i=1,2$. The map $\mathcal{R}$ is defined by
\begin{align*}
    \mathcal{R}(\Phi)= 
    \begin{pmatrix} \theta e+ \alpha u + \bb{P}\big[|u|^{r-1} u\big] \\[0.5ex] (u\cdot\nabla) \phi \end{pmatrix},
\end{align*}
where $\Phi=(u,\theta)$. The stochastic term is $g(\Phi)\dot{W}= (g_1\dot{W}_1,\, g_2\dot{W}_2)^\top$. 

By similar argument as in Section~\ref{sec:scbf}, we infer that $\mathcal{B}$ fulfils condition (B). Next, we verify condition (R) for the map $\mathcal{R}$. Firstly, \eqref{equ:Rv v pos} is satisfied with $\mathcal{R}(\Phi)=\mathcal{S}(\Phi)+ \mathcal{F}(\Phi)$, where
\begin{align*}
    \mathcal{S}(\Phi)= 
    \begin{pmatrix} \alpha u + \bb{P}\big[|u|^{r-1} u\big] \\[0.5ex] 0 \end{pmatrix},
    \quad 
    \mathcal{F}(\Phi)= 
    \begin{pmatrix} \theta e \\[0.5ex] (u\cdot\nabla) \phi \end{pmatrix}.
\end{align*}
Since $\mathcal{S}$ has the same form as \eqref{equ:Rv scbf}, it remains to verify~\eqref{equ:est R v in v}, \eqref{equ:est Rv1 minus v2}, and \eqref{equ:est R vv} for the map $\mathcal{F}$. Firstly, by H\"older's inequality and Sobolev embedding,
\begin{align*}
    \abs{\mathcal{F}(\Phi)}^2
    &\leq 
    C \abs{\theta}^2 + C \abs{u}_{L^4}^2 \abs{\nabla \phi}_{L^4}^2
    \leq
    C\norm{\Phi}^2,
\end{align*}
which implies \eqref{equ:est R v in v}.
Next, by H\"older's inequality and the Sobolev embedding, we have
\begin{align*}
    \abs{\inpro{\mathcal{F}(\Phi_1)}{\Phi_2}}
    &\leq
    C \abs{\theta_1} \abs{u_2} + C \abs{u_1}_{L^4} \abs{\nabla \phi}_{L^4} \abs{\theta_2}
    \\
    &\leq
    C \left(\norm{u_1}+ \norm{\theta_1} \right) \left(\abs{u_2}+ \abs{\theta_2} \right)
    \leq
    C \norm{\Phi_1} \abs{\Phi_2}.
\end{align*}
Inequality \eqref{equ:est Rv1 minus v2} can be shown similarly. Finally, by interpolation and Gagliardo--Nirenberg inequalities, we infer
\begin{align*}
    \abs{\mathcal{F}(\Phi)}_{\mathrm{D}(\mathcal{A}^{\frac18})}^2
    \leq
    \abs{\mathcal{F}(\Phi)}_{\mathrm{D}(\mathcal{A}^{\frac14})}^2
    &\leq
    C \abs{\mathcal{F}(\Phi)} \norm{\mathcal{F}(\Phi)}
    \\
    &\leq
    C\norm{\Phi} \big(\norm{\theta}+ \abs{\nabla u}_{L^4} \abs{\nabla \phi}_{L^4} + \abs{u}_{L^\infty} \abs{\phi}_{H^2} \big)
    \\
    &\leq
    C \norm{\Phi} \abs{\mathcal{A}\Phi},
\end{align*}
which implies \eqref{equ:est R vv}, thus completing the verification of condition (R). The estimate \eqref{equ:ass global} is easily seen to hold when $d=2$ and $r\in [1,3]$ by a similar argument as in \eqref{equ:Rv global}.

\subsection{Stochastic convective dynamo system}\label{sec:dynamo}

The dynamo theory~\cite{Els56} is a widely accepted mechanism for the creation and the maintenance of planetary magnetic field. It asserts that the cosmic magnetic field is generated and maintained by the turbulent motion of conducting fluid. 

Let $\mathscr{O}= \bb{T}^{d-1}\times [0,h]$ be a cylindrical domain, where $h>0$ and $d\in \{2,3\}$. Let $\{e_1,\ldots,e_d\}$ be the standard basis of $\bb{R}^d$. We define the following fields:
\begin{enumerate}[(i)]
    \item $u:[0,T]\times \mathscr{O}\to \bb{R}^d$ is the fluid velocity, i.e. $u=(u_1,\ldots,u_d)$,
    \item $B: [0,T]\times \mathscr{O}\to \bb{R}^d$ is the magnetic flux density, i.e. $B=(B_1,\ldots,B_d)$,
    \item $\theta :[0,T] \times \mathscr{O} \to \bb{R}$ is the temperature profile.
\end{enumerate}
The dynamics of the fluid motion, the magnetic field, and the temperature profile are described by the stochastic convective dynamo system~\cite{CatEmoWei03, CatHug06, Tob21}:
\begin{subequations}\label{equ:scds}
	\begin{alignat}{2}
		&\partial_t u - \nu_1 \Delta u + \sigma e_d \times u + (u\cdot\nabla)u - (B\cdot\nabla) B + \nabla \left(p+\frac12 \abs{B}^2\right) + \theta e_d = g_1 \dot{W}_1
		\; && \quad\text{in $\mathscr{O}$,}
		\label{equ:scds a1}
		\\[1ex]
        &\partial_t B - \nu_2 \Delta B + (u\cdot\nabla)B - (B\cdot\nabla)u = g_2 \dot{W}_2
		\; && \quad\text{in $\mathscr{O}$,}
		\label{equ:scds a2}
		\\[1ex]
        \label{equ:scds a3}
        &\partial_t \theta - \nu_3 \Delta \theta + (u\cdot\nabla)\theta + e_d\cdot u = g_3 \dot{W}_3
        \; && \quad\text{in $\mathscr{O}$,}
        \\[1ex]
		&\divg u = 0,\; \divg B=0
		\; && \quad\text{in $\mathscr{O}$,}
		\label{equ:scds b}
		\\[1ex]
		&u(0,x)= u_0(x) ,\; B(0,x)=B_0(x), \; \theta(0,x)=\theta_0(x)
		\; && \quad\text{in $\mathscr{O}$,}
		\label{equ:scds c}
		\\[1ex]
		&u_d=B_d=0,\; \theta=0
		\; && \quad\text{on $\Gamma_d$.}
		\label{equ:scds d}
	\end{alignat}
\end{subequations}
The system \eqref{equ:scds} is equipped with a periodic boundary condition in the lateral parts. Here, $\Gamma_d:= \{x: x_d=0 \text{ or } h\}$ is the top and the bottom boundaries. The constant $\sigma$ is the Coriolis force parameter, which is zero when $d=2$.

The functional setting for this system will be described next. Let $H^1_{\mathrm{d}}(\mathscr{O})$ denote the closure in $H^1(\mathscr{O})$ of the space of compactly supported smooth functions that are periodic in the lateral direction. Define the spaces:
\begin{align*}
    V_1 &= V_2 := \{v\in H^1_\mathrm{d}(\mathscr{O}): \divg v=0\},
    \\
    V_3 &:= H^1_\mathrm{d}(\mathscr{O})
    \\
    H_1 &= H_2 := \{v\in L^2(\mathscr{O}): \divg v=0 \;\text{ in $\mathscr{O}$ and } \left. v\cdot n \right|_{\Gamma_d}= 0\},
    \\
    H_3 &:= L^2(\mathscr{O}),
    \\
    \mathcal{V}&:= V_1\times V_2 \times V_3,
    \\
    \mathcal{H}&:= H_1\times H_2 \times H_3.
\end{align*}
For $i=1,2$, let $A_i= -\nu_i \bb{P}\Delta$ be the (periodic-Dirichlet) Stokes operator, and let $A_3= -\nu_3 \Delta$ be the Laplacian operator with the same boundary conditions. Let $\mathcal{A}= A_1 \oplus A_2 \oplus A_3$ with domain $\mathrm{D}(\mathcal{A})= \mathcal{V} \cap \left[H^2(\mathscr{O})\right]^3$ acting on $\Phi=(u,B,\theta)$. More generally, it is known that $\mathrm{D}(\mathcal{A}^{\frac{s}{2}})= \big[ H^s(\mathscr{O})\big]^3$ for $s\in [0,\frac12)$ and $\mathrm{D}(\mathcal{A}^{\frac12})= \mathcal{V}$; see~\cite{IftRau01} for further properties of these function spaces and operators. The map $\mathcal{B}$ is given by
\begin{align*}
    \mathcal{B}(\Phi_1,\Phi_2)= 
    \begin{pmatrix} \bb{P}\big[(u_1\cdot\nabla)u_2- (B_1\cdot\nabla)B_2 \big] \\[0.5ex] (u_1\cdot\nabla) B_2- (B_1\cdot\nabla)u_2 \\[0.5ex]
    (u_1\cdot\nabla)\theta_2\end{pmatrix}
\end{align*}
for $\Phi_i=(u_i,B_i,\theta_i)\in \mathcal{V}$, where $i=1,2$. The map $\mathcal{R}$ is defined by
\begin{align*}
    \mathcal{R}(\Phi)= 
    \begin{pmatrix} \sigma \bb{P}\big[e_d\times u\big] + \theta e_d\\[0.5ex] 0 \\[0.5ex] e_d\cdot u \end{pmatrix},
\end{align*}
where $\Phi=(u,B,\theta)$. The stochastic term is $g(\Phi)\dot{W}= (g_1\dot{W}_1,\, g_2\dot{W}_2, g_3\dot{W}_3)^\top$. To show that condition (B) and condition (R) in the assumptions are satisfied, we can proceed in a similar manner as in Sections~\ref{sec:scbf} and~\ref{sec:sbc}.

\subsection{Stochastic micropolar and magneto-micropolar fluid systems}\label{sec:mmp}

The theory of micropolar fluids~\cite{Eri66} describes the behaviour of fluids that exhibit rotational effects and micro-rotational inertia, which include blood and liquid crystal. Taking the electromagnetic effects into consideration, we obtain the stochastic magneto-micropolar fluid system~\cite{InoMatOta03, Yam14, Yam16m}, which reads:
\begin{subequations}\label{equ:smmf}
	\begin{alignat}{2}
		&\partial_t u - (\mu+\chi) \Delta u + (u\cdot\nabla)u - (B\cdot\nabla) B + \nabla \left(p+\frac12 \abs{B}^2\right) - \chi \,\curl w= g_1 \dot{W}_1
		\; && \quad\text{in $\mathscr{O}$,}
		\label{equ:smmf a1}
        \\[1ex]
        \label{equ:smmf a2}
        &\partial_t w - \gamma \Delta w - (\alpha+\beta)\nabla(\divg w) + 2\chi w + (u\cdot\nabla)w - \chi \,\curl u = g_2 \dot{W}_2
        \; && \quad\text{in $\mathscr{O}$,}
		\\[1ex]
        \label{equ:smmf a3}
        &\partial_t B - \nu \Delta B + (u\cdot\nabla)B - (B\cdot\nabla)u = g_3\dot{W}_3
        \; && \quad\text{in $\mathscr{O}$,}
        \\[1ex]
		&\divg u = \divg B= 0
		\; && \quad\text{in $\mathscr{O}$,}
		\label{equ:smmf b}
		\\[1ex]
		&u(0,x)= u_0(x),\; w(0,x)=w_0(x), \; B(0,x)= B_0(x)
		\; && \quad\text{in $\mathscr{O}$,}
		\label{equ:smmf c}
		\\[1ex]
		&u = w = 0,
		\; && \quad\text{on $\partial\mathscr{O}$,}
		\label{equ:smmf d}
        \\[1ex]
        \label{equ:smmf e}
        &B\cdot n=0, \; \mathrm{curl}\; B\times n=0
        \; && \quad \text{on $\partial\mathscr{O}$.}
	\end{alignat}
\end{subequations}
Here, $u,w,B$ are vector fields representing the fluid velocity, the micro-rotational vector of the fluid, and the magnetic flux density, respectively. For physical reasons, the coefficients $\mu, \chi,\gamma,\alpha,\beta,\nu$ are assumed to satisfy $\min\{\mu,\chi,\gamma,\nu,\alpha+\beta+\gamma\} >0$. This system can also be further coupled with an equation of the form \eqref{equ:scds a3} describing the temperature dynamics~\cite{KalLanLuk19, XuQiaZha21}.

The existence of strong solutions to the deterministic magneto-micropolar fluid system is shown in~\cite{Roj97, Yam05}. In the stochastic case, the existence of martingale weak solutions to the system is proven in~\cite{Yam14, Yam16m}. This paper establishes the existence and uniqueness of local pathwise strong solutions to the stochastic magneto-micropolar fluid system for $d\leq 3$. This pathwise strong solution is shown to be global when $d\leq 2$.

The functional setting for this problem will be described next. Define:
\begin{align*}
    V_1 &:= \{v\in H^1_0(\mathscr{O}): \divg v=0\},
    \\
    V_2 &:= H^1_0(\mathscr{O}),
    \\
    V_3 &:= \{v\in H^1(\mathscr{O}): \divg u=0 \text{ in $\mathscr{O}$, } \left. u\cdot n \right|_{\partial \mathscr{O}}= 0 \,\text{ and } \left.\mathrm{curl}\, u\times n\right|_{\partial\mathscr{O}}=0\},
    \\
    H_1 &:= \{v\in L^2(\mathscr{O}): \divg v=0 \;\text{ and } \left. v\cdot n \right|_{\partial \mathscr{O}}= 0\},
    \\
    H_2 &:= L^2(\mathscr{O})
    \\
    H_3 &:= H_1
    \\
    \mathcal{V} &:= V_1 \times V_2 \times V_3,
    \\
    \mathcal{H} &:= H_1 \times H_2 \times H_3.
\end{align*}
Let $\bb{P}:L^2(\mathscr{O})\to H_1$ be the Leray--Helmholtz projection, $\Delta_D$ be the Dirichlet Laplacian, and $\Delta_H$ be the Hodge Laplacian. Define the operators:
\begin{align*}
    A_1 &:= -(\mu+\chi)\bb{P}\Delta_D, \text{ where } 
    \mathrm{D}(A_1)= V_1 \cap H^2(\mathscr{O}),
    \\
    A_2 &:= -\gamma\Delta_D  -(\alpha+\beta)\nabla \divg,\text{ where } \mathrm{D}(A_2)= V_2 \cap H^2(\mathscr{O}),
    \\
    A_3 &:= -\nu \Delta_H, \text{ where } \mathrm{D}(A_3)= V_3 \cap H^2(\mathscr{O}),
    \\
    \mathcal{A} &:= A_1\oplus A_2 \oplus A_3, \text{ where } \mathrm{D}(\mathcal{A})= A_1\times A_2\times A_3.
\end{align*}
The map $\mathcal{B}$ is given by
\begin{align*}
    \mathcal{B}(\Phi_1,\Phi_2)= 
    \begin{pmatrix} \bb{P}\big[(u_1\cdot\nabla)u_2- (B_1\cdot\nabla)B_2 \big] \\[0.5ex] (u_1\cdot\nabla) w_2
    \\[0.5ex]
    (u_1\cdot\nabla)B_2- (B_1\cdot\nabla) u_2 \end{pmatrix}
\end{align*}
for $\Phi_i=(u_i,w_i,B_i)\in \mathcal{V}$, where $i=1,2$. The map $\mathcal{R}$ is defined by
\begin{align*}
    \mathcal{R}(\Phi)= 
    \begin{pmatrix} -\chi \,\curl w \\[0.5ex] 2\chi w-\chi \,\curl u \\[0.5ex] 0 \end{pmatrix},
\end{align*}
where $\Phi=(u,w,B)$. The stochastic term is $g(\Phi)\dot{W}= (g_1\dot{W}_1,\, g_2\dot{W}_2, g_3\dot{W}_3)^\top$. As before, we can verify that our standing assumptions are valid for the operators defined above.

\subsection{Stochastic diffusive tropical climate model}\label{sec:climate}

A tropical climate model of the equatorial atmosphere is proposed by Frierson--Majda--Pauluis in~\cite{FriMajPau04} by performing a Galerkin truncation to the hydrostatic Boussinesq equation. Since then, several physically reasonable modifications of the model have been proposed, including adding diffusion or damping terms to some or all of the equations~\cite{LiTi16, Ye17}. Here, we consider the stochastic diffusive tropical climate model:
\begin{subequations}\label{equ:stc}
	\begin{alignat}{2}
		&\partial_t u - \nu_1 \Delta u + (u\cdot\nabla)u + \nabla p + \divg(v\otimes v) = g_1 \dot{W}_1
		\; && \quad\text{in $\mathscr{O}$,}
		\label{equ:stc a1}
        \\[1ex]
        \label{equ:stc a2}
        &\partial_t v - \nu_2 \Delta v + (u\cdot\nabla) v + \nabla\theta + (v\cdot\nabla)u = g_2 \dot{W}_2
        \; && \quad\text{in $\mathscr{O}$,}
		\\[1ex]
        \label{equ:stc a3}
        &\partial_t \theta - \nu_3 \Delta \theta + (u\cdot\nabla)\theta + \divg v = g_3\dot{W}_3
        \; && \quad\text{in $\mathscr{O}$,}
        \\[1ex]
		&\divg u = 0
		\; && \quad\text{in $\mathscr{O}$,}
		\label{equ:stc b}
		\\[1ex]
		&u(0,x)= u_0(x), \; v(0,x)= v_0(x), \; \theta(0,x)=\theta_0(x)
		\; && \quad\text{in $\mathscr{O}$,}
		\label{equ:stc c}
		\\[1ex]
		&u = v = 0,\; \theta=0,
		\; && \quad\text{on $\partial\mathscr{O}$,}
		\label{equ:stc d}
	\end{alignat}
\end{subequations}
where $\nu_1,\nu_2,\nu_3>0$ are diffusivity constants, and
\begin{enumerate}[(i)]
    \item $u:[0,T]\times \mathscr{O}\to \bb{R}^d$ is the barotropic mode,
    \item $v: [0,T]\times \mathscr{O}\to \bb{R}^d$ is the first baroclinic mode of the velocity vector,
    \item $\theta :[0,T] \times \mathscr{O} \to \bb{R}$ is the temperature profile.
\end{enumerate}

In the deterministic case, the system~\eqref{equ:stc} is widely studied. Global existence of strong solutions for $d=2$ is established in~\cite{LiTi16}, even assuming $\nu_3=0$. Global existence of solutions for $d=2$ with $-\nu_1\Delta u$ replaced by the fractional Laplacian is shown in~\cite{Ye17}; see also~\cite{DonWanWu19, DonWuYe19}. For $d=3$, the system with additional damping terms is studied in~\cite{YuaChe21}. In the stochastic case, local existence and uniqueness of strong solutions for the problem when $d=2$ is shown in~\cite{YinYan23}. In this paper, we establish the local existence and uniqueness of pathwise strong solutions for $d\leq 3$, which is global when $d=2$, thus extending known results.

We describe the functional setting for this problem as follows. Define:
\begin{align*}
    V_1 &:= \{v\in H^1_0(\mathscr{O}): \divg v=0\},
    \\
    V_2 &= V_3 := H^1_0(\mathscr{O}),
    \\
    H_1 &:= \{v\in L^2(\mathscr{O}): \divg v=0 \;\text{ and } \left. v\cdot n \right|_{\partial \mathscr{O}}= 0\},
    \\
    H_2 &= H_3 := L^2(\mathscr{O})
    \\
    \mathcal{V} &:= V_1 \times V_2 \times V_3,
    \\
    \mathcal{H} &:= H_1 \times H_2 \times H_3.
\end{align*}
Let $A_1= -\nu_1 \bb{P}\Delta$ be the Dirichlet--Stokes operator, $A_2=-\nu_2\Delta$, and $A_3=-\nu_3 \Delta$, where $\Delta$ is Dirichlet Laplacian operator. Let $\mathcal{A}= A_1 \oplus A_2 \oplus A_3$ with domain $\mathrm{D}(\mathcal{A})= \mathcal{V} \cap \left[H^2(\mathscr{O})\right]^3$ acting on $\Phi=(u,v,\theta)$. The map $\mathcal{B}$ is given by
\begin{align*}
    \mathcal{B}(\Phi_1,\Phi_2)= 
    \begin{pmatrix} \bb{P}\big[(u_1\cdot\nabla)u_2+ \divg(v_1\otimes v_2) \big] \\[0.5ex] (u_1\cdot\nabla) v_2 + (v_1\cdot\nabla) u_2
    \\[0.5ex]
    (u_1\cdot\nabla)\theta_2 \end{pmatrix}
\end{align*}
for $\Phi_i=(u_i,v_i,\theta_i)\in \mathcal{V}$, where $i=1,2$. The map $\mathcal{R}$ is defined by
\begin{align*}
    \mathcal{R}(\Phi)= 
    \begin{pmatrix} 0 \\[0.5ex] \nabla\theta \\[0.5ex] \divg v \end{pmatrix},
\end{align*}
where $\Phi=(u,w,B)$. The stochastic term is $g(\Phi)\dot{W}= (g_1\dot{W}_1,\, g_2\dot{W}_2, g_3\dot{W}_3)^\top$. We can verify that our standing assumptions are valid for the operators defined above.

%

\section{Existence and uniqueness of a local strong solution}\label{sec:loc sol}

Suppose that we fix a stochastic basis $\mathcal{S}:=(\Omega, \mathcal{F}, \mathbb{F}, \mathbb{P}, \{\beta_k\}_{k\in \bb{N}})$, where $\bb{F}= \{\mathcal{F}_t\}_{t\geq 0}$ is a filtration satisfying the usual conditions and $\{\beta_k\}_{k\in \bb{N}}$ is a sequence of independent Brownian motions adapted to this filtration. We have the following definitions of weak and strong pathwise solutions of~\eqref{equ:equation}.

\begin{definition}[local weak pathwise solution]\label{def:weak sol}
The pair $(\Phi,\tau)$ is a \emph{local weak pathwise solution} of \eqref{equ:equation} with initial data $\Phi_0\in L^\infty(\Omega;\mathcal{H})$ if $\tau$ is a strictly positive stopping time and $\Phi(\,\cdot\wedge \tau)$ is a predictable process in $\mathcal{V}'$ such that
\begin{align}\label{equ:reg local weak}
	\Phi(\,\cdot\wedge \tau)\in L^2\big(\Omega; C_\mathrm{w}([0,\infty); \mathcal{H})\big), 
	\quad
	\Phi \one_{t\leq \tau} \in L^2\big(\Omega; L^2_{\mathrm{loc}}([0,\infty); \mathcal{V})\big),
\end{align}
and that for any $t>0$,
\begin{align}\label{equ:int equal}
	\Phi(t\wedge\tau) + \int_0^{t\wedge\tau} \big(\mathcal{A}\Phi+ \mathcal{B}(\Phi) + \mathcal{R}(\Phi)\big)\, \ds
	=
	\Phi_0 + \int_0^{t\wedge\tau} g(\Phi)\, \dW \quad \text{in $\mathcal{V}'$}.
\end{align}
\end{definition}

\begin{definition}[local strong pathwise solution]\label{def:strong sol}
A local weak pathwise solution $(\Phi,\tau)$ is said to be a \emph{local strong pathwise solution} of \eqref{equ:equation} with initial data $\Phi_0\in L^\infty(\Omega;\mathcal{V})$ if $\Phi(\,\cdot\wedge \tau)$ is a predictable process in $\mathcal{H}$ such that
\begin{align}\label{equ:reg local strong}
	\Phi(\,\cdot\wedge \tau)\in L^2\big(\Omega; C([0,\infty); \mathcal{V})\big), 
	\quad
	\Phi \one_{t\leq \tau} \in L^2\big(\Omega; L^2_{\mathrm{loc}}([0,\infty); \mathrm{D}(\mathcal{A}))\big),
\end{align}
and that for any $t>0$, the equality~\eqref{equ:int equal} {holds} in $\mathcal{H}$.
\end{definition}

\begin{definition}[maximal strong pathwise solution]\label{def:global strong sol}
The pair $(\Phi,\xi)$ is said to be a \emph{maximal strong pathwise solution} if there exists an increasing sequence of stopping times $\{\tau_n\}_{n\in \bb{N}}$ with $\tau_n \uparrow \xi$ a.s., such that each pair $(\Phi,\tau_n)$ is a local pathwise strong solution and
\begin{align*}
	\sup_{t<\xi} \norm{\Phi(t)}^2 + \int_0^\xi \abs{\mathcal{A}\Phi(s)}^2 \ds = \infty \;\text{ a.s. on the set } \{\xi<\infty\}.
\end{align*}
The sequence $\{\tau_n\}_{n\in \bb{N}}$ is said to \emph{announce} (finite time blow up at) $\xi$. 

Such a maximal strong pathwise solution $(\Phi,\xi)$ is \emph{global} if $\bb{P}\left[\xi=\infty \right]=1$.
\end{definition}

\begin{remark}\label{rem1}
Clearly, the following holds:
\begin{enumerate}
\item $\xi$ is a stopping time;
\item we can choose the announcing sequence $\{\tau_n\}$ such that
\end{enumerate}
\begin{align*}
	\sup_{t\leq \tau_n} \norm{\Phi(t)}^2 + \int_0^{\tau_n} \abs{\mathcal{A}\Phi(s)}^2 \ds = n \;\text{ a.s. on the set } \{\xi<\infty\}.
\end{align*}
\end{remark}
To show the existence of a solution to \eqref{equ:equation}, we use the Galerkin method and a compactness argument utilising Theorem~\ref{the:compactness}. Let $\mathcal{H}_n:= \text{span}\{e_1,e_2,\ldots,e_n\}$, where $\{e_k\}_{k\in \bb{N}}$ is an orthonormal basis for $\mathcal{H}$ consisting of eigenfunctions of $\mathcal{A}$. Let $P_n:\mathcal{H}\to \mathcal{H}_n$ be the projection operator and let $Q_n=I-P_n$, where $I$ is the identity operator.

The Galerkin approximation of order $n$ for \eqref{equ:equation} is the following finite-dimensional problem in $\mathcal{H}_n$: find an adapted process $\Phi^n \in C([0,T];\mathcal{H}_n)$ such that
\begin{equation}
\begin{aligned}\label{equ:galerkin}
	\mathrm{d}\Phi^n + \big(\mathcal{A} \Phi^n + P_n \mathcal{B}(\Phi^n) + P_n \mathcal{R}(\Phi^n) \big) \dt
	&=
	\sum_{k=1}^\infty P_n g_k(\Phi^n)\, \mathrm{d}\beta_k,
	\\
	\Phi^n(0) &= \Phi_0^n:= P_n \Phi_0.
\end{aligned}
\end{equation}
We choose initial approximations $\Phi_0^n\in\mathcal H_n$ such that $\Phi_0^n\to \Phi_0$ in $\mathcal{H}$ a.s. and, for some deterministic constant $M>1$ independent of $n$,
\begin{align}\label{equ:M}
	\norm{\Phi_0^n}^2\leq M,
	\qquad \mathbb P\text{-a.s.}
\end{align}

For fixed $n$, \eqref{equ:galerkin} is a finite-dimensional SDE on
$\mathcal H_n$. Since all norms are equivalent on $\mathcal H_n$, the
projected coefficients are locally Lipschitz, while the one-sided Lyapunov estimate
below prevents finite-time explosion. Hence, the standard finite-dimensional
SDE theory yields a unique solution on $[0,T]$; see~\cite[Chapter~10]{ChuWil14}. For a detailed verification of these finite-dimensional conditions in a related Galerkin setting, see~\cite{GolSoeTra26}.

\begin{lemma}
	For each $n\in \bb{N}$, there exists a constant $C_n>0$ such that, for all $v_n\in \mathcal{H}_n$,
	\begin{align*}
		\inpro{-\mathcal{A}v_n- P_n\mathcal{B}(v_n)- P_n\mathcal{R}(v_n)}{v_n} + \sum_{k=1}^\infty \abs{P_n g_k(v_n)}^2 \leq C_n\big(1+ \abs{v_n}^2 \big),
	\end{align*}
	where $C_n$ is a constant, possibly depending on $n$.
\end{lemma}

\begin{proof}
	Since $v_n\in\mathcal H_n$ and $P_n$ is the orthogonal projection onto
	$\mathcal H_n$, the projections may be removed in the inner products with
	$v_n$. Thus, by the positivity of $\mathcal A$, the cancellation property
	\eqref{equ:B antisym}, and the condition~\eqref{equ:Rv v pos},
	\[
	\inpro{
		-\mathcal A v_n
		-
		P_n\mathcal B(v_n)
		-
		P_n\mathcal R(v_n)}{v_n}
	\leq
	C_n(1+\abs{v_n}^2),
	\]
	where the last inequality follows from the equivalence of norms on
	$\mathcal H_n$.
	Moreover, \eqref{equ:cond G} gives
	\[
	\sum_{k=1}^\infty \abs{P_n g_k(v_n)}^2
	\leq
	C(1+\abs{v_n}^2).
	\]
	This proves the claim.
\end{proof}

For any $T>0$, we introduce the following class of admissible stopping times:
\begin{align}\label{equ:Tn MT}
	\mathcal{T}_n^{M,T}= \left\{ \tau\leq T: \sup_{t\in [0,\tau]} \norm{\Phi^n(t)}^2 + \int_0^\tau \abs{\mathcal{A}\Phi^n(s)}^2 \ds  \leq 4M \right\}.
\end{align}
The following propositions provide the two compactness inputs required for
Theorem~\ref{the:compactness}. Their proofs are based on the localisation
and comparison strategy of~\cite{GlaTem11,GlaZia09}, adapted here
to the present abstract setting and nonlinear structure.

\begin{proposition}\label{pro:hypotheses}
Let $\{\Phi^n\}_{n\in \bb{N}}$ be a sequence of solutions to \eqref{equ:galerkin}. Suppose that there exists a deterministic constant $M$ such that $\Phi_0\in \mathcal{V}$. Let $\mathcal{T}_{n,m}^{M,T}:= \mathcal{T}_n^{M,T} \cap \mathcal{T}_m^{M,T}$, where $\mathcal{T}_n^{M,T}$ was defined in \eqref{equ:Tn MT}.
Then the following statements hold true:
\begin{enumerate}[(i)]
	\item For any $T>0$ and $M>1$, there exists a constant $\widetilde{C}$ independent of $n$ such that
	\begin{align}\label{equ:E bound}
		\sup_{\tau\in \mathcal{T}_{n}^{M,T}} \bb{E} \left[ \sup_{t\in [0,\tau]} \norm{\Phi^n(t)}^2 + \int_0^\tau \abs{\mathcal{A}\big(\Phi^n (s)\big)}^2 \ds \right] \leq \widetilde{C}.
	\end{align}
	\item For any $T>0$ and $M>1$,
	\begin{align}\label{equ:lim sup cauchy}
		\lim_{n\to\infty} \sup_{m>n}\, \sup_{\tau\in \mathcal{T}_{m,n}^{M,T}} \bb{E} \left[ \sup_{t\in [0,\tau]} \norm{\Phi^m(t)-\Phi^n(t)}^2 + \int_0^\tau \abs{\mathcal{A}\big(\Phi^m(s)-\Phi^n (s)\big)}^2 \ds \right] = 0. 
	\end{align}
\end{enumerate}
\end{proposition}

\begin{proof}
First, we prove \eqref{equ:E bound}. By It\^o's lemma applied to $\norm{\Phi^n}^2$, we have
\begin{align}\label{equ:d Phin}
	\mathrm{d} \norm{\Phi^{n}}^2
	+
	2 \abs{\mathcal{A} \Phi^{n}}^2 \dt 
	&=
	2\inpro{P_n \mathcal{B}(\Phi^n)}{\mathcal{A}\Phi^{n}} \dt
	+ 
	2\inpro{ P_n \mathcal{R}(\Phi^n)}{\mathcal{A} \Phi^{n}} \dt
	\nonumber\\
	&\quad
	+
	\sum_{k=1}^\infty \norm{P_n g_k(\Phi^n)}^2 \dt
	+
	2\sum_{k=1}^\infty \inpro{P_n g_k(\Phi^n)}{\mathcal{A} \Phi^{n}} \mathrm{d}\beta_k.
\end{align}
Let $\tau\in \mathcal{T}_n^{M,T}$ and let $\tau_a$ and $\tau_b$ be stopping times such that $0\leq \tau_a\leq \tau_b\leq \tau$. Integrating \eqref{equ:d Phin} over $[\tau_a,r]$, then taking supremum over $r\in [\tau_a,\tau_b]$ and applying the expected value, we have by the triangle inequality,
\begin{align}\label{equ:I1a to I4a}
	&\bb{E} \left[ \sup_{t\in [\tau_a,\tau_b]} \norm{\Phi^{n}(t)}^2 \right] 
	+
	2\bb{E} \left[ \int_{\tau_a}^{\tau_b} \abs{\mathcal{A}\Phi^{n}(t)}^2 \dt \right]
	\nonumber\\
	&\leq
	\bb{E}\left[ \norm{\Phi^{n}(\tau_a)}^2 \right]
	+
	2\bb{E} \left[ \int_{\tau_a}^{\tau_b} \abs{\inpro{P_n \mathcal{B}(\Phi^n)}{\mathcal{A} \Phi^{n}}} \dt \right]
	+
	\bb{E} \left[ \int_{\tau_a}^{\tau_b} \abs{\inpro{P_n \mathcal{R}(\Phi^n)}{\mathcal{A} \Phi^{n}}} \dt \right]
	\nonumber\\
	&\quad
	+
	\bb{E} \left[ \int_{\tau_a}^{\tau_b} \sum_{k=1}^\infty \norm{P_n g_k(\Phi^n)}^2 \dt \right]
	+
	2 \bb{E} \left[ \sup_{r\in [\tau_a,\tau_b]} \left| \sum_{k=1}^\infty \int_{\tau_a}^r \inpro{P_n g_k(\Phi^n)}{\mathcal{A} \Phi^{n}}\, \mathrm{d}\beta_k \right| \right]
	\nonumber\\
	&=: \bb{E}\left[ \norm{\Phi^{n}(\tau_a)}^2 \right] +J_1+ \ldots+ J_4.
\end{align}
We now estimate each term in the last line. For the first term, by \eqref{equ:est B with A} and Young's inequality, we have 
\begin{align}\label{equ:J1}
	J_1
	&\leq
	C \bb{E}\left[ \int_{\tau_a}^{\tau_b} \norm{\Phi^{n}}^6 \dt \right] 
	+ \frac14 \bb{E}\left[\int_{\tau_a}^{\tau_b} \abs{\mathcal{A} \Phi^{n}}^2 \dt \right].
\end{align}
For the second term, by \eqref{equ:est R v in v} and Young's inequality, we have for some $r\geq 1$,
\begin{align}\label{equ:J2}
	J_2
	&\leq
	C \bb{E} \left[\int_{\tau_a}^{\tau_b} \left(1+ \norm{\Phi^n}^{2(r-1)} \right) \norm{\Phi^{n}}^2 \dt\right] 
	+ \frac14 \bb{E}\left[ \int_{\tau_a}^{\tau_b}  \abs{\mathcal{A} \Phi^{n}}^2\dt \right].
\end{align}
For the term $J_3$, by the Lipschitz property of $g$, we have
\begin{align*}
	J_3
	\leq
	C\bb{E} \left[\int_{\tau_a}^{\tau_b} \left(1+\norm{\Phi^n}^2\right) \dt\right].
\end{align*}
For the final term, by the Burkholder--Davis--Gundy inequality and the Lipschitz property of $g$, we obtain
\begin{align*}
	J_4
	&\leq
	C \bb{E} \left[ \int_{\tau_a}^{\tau_b} \norm{\Phi^{n}}^2 \sum_{k=1}^\infty \norm{P_n g_k(\Phi^n)}^2 \dt \right]^{\frac12}
	\leq
	C \bb{E}\left[ \int_{\tau_a}^{\tau_b} \left(1+\norm{\Phi^{n}}^4 \right) \dt \right]^{\frac12},
\end{align*}
Altogether, continuing from \eqref{equ:I1a to I4a} and rearranging the terms, we infer that
\begin{align*}
	&\bb{E} \left[ \sup_{t\in [\tau_a,\tau_b]} \norm{\Phi^{n}(t)}^2 \right] 
	+
	\bb{E} \left[ \int_{\tau_a}^{\tau_b} \abs{\mathcal{A}\Phi^{n}(t)}^2 \dt \right]
	\\
	&\leq
	C+ \bb{E}\left[ \norm{\Phi^{m,n}(\tau_a)}^2 \right]
	+ 
	C \bb{E} \left[\int_{\tau_a}^{\tau_b} \left(1+ \norm{\Phi^n}^{2(r-1)} \right) \norm{\Phi^{n}}^2 \dt\right],
\end{align*}
where $C$ depends on $M$, but is independent of $n$, $\tau_a$, or $\tau_b$. The required inequality then follows from the Gronwall lemma for stochastic processes (Lemma~\ref{lem:gronwall}).

Next, we show \eqref{equ:lim sup cauchy}. For $m>n$, let $\Phi^{m,n}:= \Phi^m-\Phi^n$. We subtract equation \eqref{equ:galerkin} for $\Phi^n$ from that for $\Phi^m$ to obtain
	\begin{align*}
		\mathrm{d} \Phi^{m,n} + \mathcal{A} \Phi^{m,n} \dt 
		&=
		\left[ P_n \mathcal{B}(\Phi^n)- P_m\mathcal{B}(\Phi^m) + P_n \mathcal{R}(\Phi^n) - P_m \mathcal{R}(\Phi^m) \right] \dt 
		\\
		&\quad +
		\sum_{k=1}^\infty \left[P_m g_k(\Phi^m)- P_n g_k(\Phi^n)\right] \mathrm{d}\beta_k,
		\\
		\Phi^{m,n}(0) &= (P_m-P_n)\Phi_0.
	\end{align*}
	By It\^o's lemma applied to $\norm{\Phi^{m,n}}^2$, we have 
	\begin{align}\label{equ:ito d Phi mn}
		&\mathrm{d} \norm{\Phi^{m,n}}^2
		+
		2 \abs{\mathcal{A} \Phi^{m,n}}^2 \dt 
		\nonumber\\
		&=
		2\inpro{P_n \mathcal{B}(\Phi^n)- P_m\mathcal{B}(\Phi^m)}{\mathcal{A}\Phi^{m,n}} \dt
		+ 2\inpro{ P_n \mathcal{R}(\Phi^n) - P_m \mathcal{R}(\Phi^m)}{\mathcal{A} \Phi^{m,n}} \dt
		\nonumber\\
		&\quad
		+
		\sum_{k=1}^\infty \norm{P_m g_k(\Phi^m)- P_n g_k(\Phi^n)}^2 \dt
		+
		2\sum_{k=1}^\infty \inpro{P_m g_k(\Phi^m)- P_n g_k(\Phi^n)}{\mathcal{A} \Phi^{m,n}} \mathrm{d}\beta_k.
	\end{align}
We now apply a similar argument as before. Let $\tau\in \mathcal{T}_{m,n}^{M,T}$ and let $\tau_a$ and $\tau_b$ be stopping times such that $0\leq \tau_a\leq \tau_b\leq \tau$. Integrating \eqref{equ:ito d Phi mn} over $[\tau_a,r]$, then taking supremum over $r\in [\tau_a,\tau_b]$ and applying the expected value, we obtain,
\begin{align}\label{equ:I1 to I4}
	&\bb{E} \left[ \sup_{t\in [\tau_a,\tau_b]} \norm{\Phi^{m,n}(t)}^2 \right] 
	+
	2\bb{E} \left[ \int_{\tau_a}^{\tau_b} \abs{\mathcal{A}\Phi^{m,n}(t)}^2 \dt \right]
	\nonumber\\
	&\leq
	\bb{E}\left[ \norm{\Phi^{m,n}(\tau_a)}^2 \right]
	+
	2\bb{E} \left[ \int_{\tau_a}^{\tau_b} \abs{\inpro{P_m \mathcal{B}(\Phi^m)- P_n \mathcal{B}(\Phi^n)}{\mathcal{A} \Phi^{m,n}}} \dt \right]
	\nonumber\\
	&\quad
	+
	\bb{E} \left[ \int_{\tau_a}^{\tau_b} \abs{\inpro{P_m \mathcal{R}(\Phi^m)- P_n \mathcal{R}(\Phi^n)}{\mathcal{A} \Phi^{m,n}}} \dt \right]
	+
	\bb{E} \left[ \int_{\tau_a}^{\tau_b} \sum_{k=1}^\infty \norm{P_m g_k(\Phi^m)- P_n g_k(\Phi^n)}^2 \dt \right]
	\nonumber\\
	&\quad
	+
	2 \bb{E} \left[ \sup_{r\in [\tau_a,\tau_b]} \left| \sum_{k=1}^\infty \int_{\tau_a}^r \inpro{P_m g_k(\Phi^m)- P_n g_k(\Phi^n)}{\mathcal{A} \Phi^{m,n}}\, \mathrm{d}\beta_k \right| \right]
	\nonumber\\
	&=: \bb{E}\left[ \norm{\Phi^{m,n}(\tau_a)}^2 \right] +I_1+ \ldots+ I_4.
\end{align}
Each term $I_j$, $j=1,2,3,4$, will be estimated next. Firstly, we consider the term $I_1$. Note that we can write
\begin{align*}
	P_m \mathcal{B}(\Phi^m)- P_n \mathcal{B}(\Phi^n)
	=
	P_m \mathcal{B}(\Phi^{m,n}, \Phi^m) + P_m \mathcal{B}(\Phi^n, \Phi^{m,n}) + (P_m-P_n) \mathcal{B}(\Phi^n).
\end{align*}
By \eqref{equ:est B with A}, we have 
\begin{align}\label{equ:I1 part 1}
	\abs{\inpro{\mathcal{B}(\Phi^{m,n}, \Phi^m)}{\mathcal{A} \Phi^{m,n}}}
	&\leq
	C \norm{\Phi^{m,n}} \abs{\mathcal{A} \Phi^m} \abs{\mathcal{A} \Phi^{m,n}}
	\nonumber\\
	&\leq
	 C \norm{\Phi^{m,n}}^2 \abs{\mathcal{A} \Phi^m}^2 + \frac19 \abs{\mathcal{A} \Phi^{m,n}}^2,
\end{align}
where in the last step we used Young's inequality. Similarly, we also have
\begin{align}\label{equ:I1 part 2}
	\abs{\inpro{\mathcal{B}(\Phi^{n}, \Phi^{m,n})}{\mathcal{A} \Phi^{m,n}}}
	&\leq
	C \norm{\Phi^n} \norm{\Phi^{m,n}}^{\frac12} \abs{\mathcal{A} \Phi^{m,n}}^{\frac32}
	\nonumber\\
	&\leq
	C \norm{\Phi^{m,n}}^2 \norm{\Phi^n}^4 + \frac19 \abs{\mathcal{A} \Phi^{m,n}}^2,
\end{align}
Furthermore, by Young's inequality, Lemma~\ref{lem:poincare}, and estimate~\eqref{equ:V norm B}, we obtain {for some $\alpha\in (0,\frac14)$ and $p>0$},
\begin{align}\label{equ:I1 part 3}
	\abs{\inpro{(P_m-P_n) \mathcal{B}(\Phi^n)}{\mathcal{A} \Phi^{m,n}}}
	&\leq
	C \abs{Q_n \mathcal{B}(\Phi^n)}^2
	+
	\frac19 \abs{\mathcal{A} \Phi^{m,n}}^2 
	\nonumber \\
	&\leq
	C \lambda_n^{-2\alpha} \abs{Q_n \mathcal{B}(\Phi^n)}_{\mathrm{D}(\mathcal{A}^{\alpha})}^2
	+
	\frac19 \abs{\mathcal{A} \Phi^{m,n}}^2 
	\nonumber\\
	&\leq
	C \lambda_n^{-2\alpha} \left(1+\norm{\Phi^n}^{2p}\right) \abs{\mathcal{A}\Phi^n}^2 + \frac19 \abs{\mathcal{A} \Phi^{m,n}}^2.
\end{align}
Combining \eqref{equ:I1 part 1}, \eqref{equ:I1 part 2}, and \eqref{equ:I1 part 3}, we deduce that for the term $I_1$ in \eqref{equ:I1 to I4},
\begin{align}\label{equ:est I1}
	\abs{I_1}
	&\leq
	C \bb{E}\left[ \int_{\tau_a}^{\tau_b} \norm{\Phi^{m,n}}^2 \Big( \abs{\mathcal{A} \Phi^{m}}^2 + \norm{\Phi^n}^4 \Big) 
	+ \lambda_n^{-2\alpha} \left(1+\norm{\Phi^n}^{2p}\right) \abs{\mathcal{A}\Phi^n}^2  \dt \right]
	\nonumber\\
	&\quad 
	+
	\frac13 \bb{E} \left[ \int_{\tau_a}^{\tau_b} \abs{\mathcal{A} \Phi^{m,n}}^2 \dt\right].
\end{align}
Next, we consider the term $I_2$ in \eqref{equ:I1 to I4}. To this end, we write
\begin{align*}
	P_m \mathcal{R}(\Phi^m)- P_n \mathcal{R}(\Phi^n)
	=
	P_m \left[\mathcal{R}(\Phi^m)- \mathcal{R}(\Phi^n)\right] + (P_m-P_n) \mathcal{R}(\Phi^n).
\end{align*}
By \eqref{equ:est Rv1 minus v2} and Young's inequality, we have {for some $r\geq 1$,}
\begin{align}\label{equ:I2 part 1}
	\abs{\inpro{\mathcal{R}(\Phi^m)- \mathcal{R}(\Phi^n)}{\mathcal{A} \Phi^{m,n}}}
	&\leq
	C \left(1+ \norm{\Phi^m}^{2(r-1)}+ \norm{\Phi^n}^{2(r-1)} \right) \norm{\Phi^{m,n}}^2 + \frac19 \abs{\mathcal{A} \Phi^{m,n}}^2.
\end{align}
Moreover, by an argument similar to that in~\eqref{equ:I1 part 3}, we obtain {for some $\beta\in (0,\frac14)$ and $q>0$,}
\begin{align}\label{equ:I2 part 2}
	\abs{\inpro{(P_m-P_n) \mathcal{R}(\Phi^n)}{\mathcal{A} \Phi^{m,n}}}
	&\leq
	C \abs{Q_n \mathcal{R}(\Phi^n)}^2
	+
	\frac19 \abs{\mathcal{A} \Phi^{m,n}}^2 
	\nonumber\\
	&\leq
	C \lambda_n^{-2\beta} \abs{Q_n \mathcal{R}(\Phi^n)}_{\mathrm{D}(\mathcal{A}^{\beta})}^2 + \frac19 \abs{\mathcal{A} \Phi^{m,n}}^2 
	\nonumber\\
	&\leq
	C\lambda_n^{-2\beta} \left(1+ \norm{\Phi^n}^{2q} \right) \abs{\mathcal{A} \Phi^n}^2 
	+
	\frac19 \abs{\mathcal{A} \Phi^{m,n}}^2.
\end{align}
Thus, we deduce an estimate for the term $I_2$ in~\eqref{equ:I1 to I4}:
\begin{align}\label{equ:est I2}
	\abs{I_2} 
	&\leq
	C \bb{E}\left[ \int_{\tau_a}^{\tau_b} \norm{\Phi^{m,n}}^2 \left( 1+ \norm{\Phi^m}^{2(r-1)} + \norm{\Phi^n}^{2(r-1)}  \right) \dt \right]
	\nonumber\\
	&\quad
	+
	C\bb{E} \left[ \int_{\tau_a}^{\tau_b} \lambda_n^{-2\beta} \left(1+ \norm{\Phi^n}^{2q} \right) \abs{\mathcal{A} \Phi^n}^2 \dt \right]
	+
	\frac13 \bb{E} \left[ \int_{\tau_a}^{\tau_b} \abs{\mathcal{A} \Phi^{m,n}}^2 \dt\right].
\end{align}
The term $I_3$ in \eqref{equ:I1 to I4} will be estimated next. To this end, note that by the triangle inequality, the Lipschitz property of $g$ assumed in Condition~(G), and Lemma~\ref{lem:poincare}, we have
\begin{align}\label{equ:sum PmPn g}
	\sum_{k=1}^\infty \norm{P_m g_k(\Phi^m)- P_n g_k(\Phi^n)}^2
	&\leq
	\sum_{k=1}^\infty \norm{g_k(\Phi^m)- g_k(\Phi^n)}^2 + \norm{Q_n g_k(\Phi^n)}^2
	\nonumber\\
	&\leq
	C \norm{\Phi^{m,n}}^2 + C \lambda_n^{-1} \sum_{k=1}^\infty \abs{\mathcal{A} g_k(\Phi^n)}^2
	\nonumber\\
	&\leq
	C \norm{\Phi^{m,n}}^2 + C\lambda_n^{-1} \big(1+\abs{\mathcal{A} \Phi^n}^2\big).
\end{align}
Therefore,
\begin{align}\label{equ:est I3}
	\abs{I_3}
	&\leq
	C \bb{E} \left[\int_{\tau_a}^{\tau_b} \left( \norm{\Phi^{m,n}}^2 + \lambda_n^{-1} \big(1+\abs{\mathcal{A} \Phi^n}^2\big) \right) \dt \right].
\end{align}
The term $I_4$ is a stochastic term, which we will estimate using the Burkholder--Davis--Gundy inequality and \eqref{equ:sum PmPn g} to obtain
\begin{align}\label{equ:est I4}
	\abs{I_4}
	&\leq
	C \bb{E}\left[ \int_{\tau_a}^{\tau_b} \sum_{k=1}^\infty \abs{\inpro{P_m g_k(\Phi^m)- P_n g_k(\Phi^n)}{\mathcal{A} \Phi^{m,n}}}^2 \dt\right]^{\frac12}
	\nonumber\\
	&\leq
	C \bb{E} \left[ \int_{\tau_a}^{\tau_b} \norm{\Phi^{m,n}}^2 \sum_{k=1}^\infty \norm{P_m g_k(\Phi^m)-P_n g_k(\Phi^n)}^2 \dt \right]^{\frac12}
	\nonumber\\
	&\leq
	C \bb{E}\left[ \int_{\tau_a}^{\tau_b} \norm{\Phi^{m,n}}^2 \left(\norm{\Phi^{m,n}}^2 + \lambda_n^{-1} \big(1+ \abs{\mathcal{A}\Phi^n}^2 \big) \right) \dt \right]^{\frac12}
	\nonumber\\
	&\leq
	\frac12 \bb{E}\left[ \sup_{t\in [\tau_a,\tau_b]} \norm{\Phi^{m,n}}^2 \right] 
	+
	C\bb{E} \left[\int_{\tau_a}^{\tau_b} \left( \norm{\Phi^{m,n}}^2 + \lambda_n^{-1} \big(1+ \abs{\mathcal{A}\Phi^n}^2 \big) \right) \dt \right],
\end{align}
where in the last step we applied Young's inequality. Altogether, from estimates~\eqref{equ:est I1}, \eqref{equ:est I2}, \eqref{equ:est I3}, and \eqref{equ:est I4}, we continue from \eqref{equ:I1 to I4} to infer that
\begin{align*}
	&\bb{E} \left[ \sup_{t\in [\tau_a,\tau_b]} \norm{\Phi^{m,n}(t)}^2 \right] 
	+
	\bb{E} \left[ \int_{\tau_a}^{\tau_b} \abs{\mathcal{A}\Phi^{m,n}(t)}^2 \dt \right]
	\\
	&\leq 
	C \bb{E}\left[ \norm{\Phi^{m,n}(\tau_a)}^2 \right]
	+  
	C \bb{E}\left[\int_{\tau_a}^{\tau_b} \left(\lambda_n^{-2\alpha}+\lambda_n^{-2\beta}\right) \Big(1+\norm{\Phi^n}^{2p}+ \norm{\Phi^n}^{2q} \Big) \Big(1+\abs{\mathcal{A} \Phi^n}^2 \Big) \dt \right]
	\\
	&\quad
	+
	C \bb{E}\left[ \int_{\tau_a}^{\tau_b} \norm{\Phi^{m,n}}^2 \Big(1+ \norm{\Phi^m}^{2(r-1)}+ \norm{\Phi^n}^{2(r-1)}+ \norm{\Phi^n}^4 + \abs{\mathcal{A} \Phi^{m}}^2  \Big) \dt \right]
	\\
	&\leq
	C \bb{E}\left[ \norm{\Phi^{m,n}(\tau_a)}^2 \right]
	+
	C \bb{E}\left[ \int_{\tau_a}^{\tau_b} \norm{\Phi^{m,n}}^2 \Big(1 + \abs{\mathcal{A} \Phi^{m}}^2 \Big) \dt \right]
	\\
	&\quad
	+
	C \bb{E} \left[ \int_{\tau_a}^{\tau_b} \left(\lambda_n^{-2\alpha}+\lambda_n^{-2\beta}\right) \Big(1+\abs{\mathcal{A} \Phi^n}^2 \Big) \dt \right],
\end{align*}
where in the last step we used the fact that $\tau_a,\tau_b\in \mathcal{T}_m^{M,T}\cap \mathcal{T}_n^{M,T}$ and \eqref{equ:Tn MT} to bound $\norm{\Phi^m}$ and $\norm{\Phi^n}$. Here, $C$ depends on $M$, but is independent of $m$, $n$, $\tau_a$, or $\tau_b$. By the same token, we have
\begin{align*}
	\int_0^\tau \left(1+ \abs{\mathcal{A}\Phi^m (t)}^2 \right) \dt \leq \left(M+ \norm{\Phi^0} \right)^2 + T, \; \text{ a.s.}
\end{align*}
We are now in the position to apply the Gronwall lemma for stochastic processes (Lemma~\ref{lem:gronwall}) with $b=1+\abs{\mathcal{A}\Phi^m}^2$, $u= \norm{\Phi^{m,n}}^2$, $v=\abs{\mathcal{A}\Phi^{m,n}}^2$, and $w=\big(\lambda_n^{-2\alpha}+\lambda_n^{-2\beta}\big) \big(1+\abs{\mathcal{A} \Phi^n}^2 \big)$. This results in
\begin{align*}
	\bb{E} \left[ \sup_{t\in [\tau_a,\tau_b]} \norm{\Phi^{m,n}}^2 \right] 
	+
	\bb{E} \left[ \int_{\tau_a}^{\tau_b} \abs{\mathcal{A}\Phi^{m,n}}^2 \dt \right]
	\leq
	C\bb{E} \left[\norm{\Phi^0}^2\right]
	+
	C\left(\lambda_n^{-2\alpha}+\lambda_n^{-2\beta}\right),
\end{align*}
where $C$ depends on $M$ and $T$, but is independent of $m$, $n$, or $\tau$. Thus, taking supremum over $\tau \in \mathcal{T}_{m,n}^{M,T}$, then supremum over all $m>n$, and letting $n\to\infty$, we obtain \eqref{equ:lim sup cauchy}.
 This completes the proof of the proposition.
\end{proof}

\begin{proposition}
Let $\mathcal{T}_n^{M,T}$ be the collection of stopping times defined in \eqref{equ:Tn MT}.
For any $n\in\bb{N}, \delta>0$, and stopping time $\tau$, define
\begin{align}\label{equ:event An tau}
	A_n(\tau,\delta):= \left\{ \sup_{t\in [0,\tau\wedge \delta]} \norm{\Phi^n(t)}^2 + \int_0^{\tau\wedge \delta} \abs{\mathcal{A}\Phi^n(s)}^2 \ds > 3M \right\}.
\end{align}
Then for any $T>0$,
\begin{align}\label{equ:lim sup prob}
	\lim_{\delta \to 0^+} \sup_{n\in\bb{N}}\, \sup_{\tau \in\mathcal{T}_n^{M,T}} \bb{P}\left[ A_n(\tau,\delta) \right] =0.
\end{align}
\end{proposition}

\begin{proof}
	Let $\tau\in\mathcal T_n^{M,T}$ and set
	\[
	\theta:=\tau\wedge\delta,
	\qquad
	I_n(t):=
	\int_0^t
	\abs{\mathcal A\Phi^n(s)}^2\,\mathrm ds .
	\]
	Since $\tau\in\mathcal T_n^{M,T}$, we have, for every $t\leq \theta$,
	\[
	\norm{\Phi^n(t)}^2\leq 4M \quad \text{and}
	\quad
	I_n(t)\leq 4M .
	\]
	Applying It\^o's formula to $\norm{\Phi^n}^2$ gives
	\begin{align}\label{equ:ito d Phi n small time}
		\mathrm{d} \norm{\Phi^{n}}^2
		+
		2 \abs{\mathcal{A} \Phi^{n}}^2 \dt
		&=
		\mathcal{D}_n(t)\,\dt
		+
		\mathrm d\mathcal M_n(t),
	\end{align}
	where
	\begin{align*}
	\mathcal{D}_n(t)
		&:=
		-2\inpro{\mathcal{B}(\Phi^n)}{\mathcal{A}\Phi^{n}}
		+
		2\inpro{\mathcal{R}(\Phi^n)}{\mathcal{A} \Phi^{n}}
		+
		\sum_{k=1}^\infty \norm{P_n g_k(\Phi^n)}^2,
		\\
		\mathcal M_n(t)
		&:=
		2\sum_{k=1}^\infty
		\int_0^t
		\inpro{g_k(\Phi^n(s))}{\mathcal A\Phi^n(s)}
		\,\mathrm d\beta_k(s).
	\end{align*}
	We estimate the drift term on the stopped interval $[0,\theta]$. By
	\eqref{equ:est B with A}, Young's inequality, and the bound
	$\norm{\Phi^n(t)}^2\leq 4M$, we have
	\[
	\abs{\inpro{\mathcal B(\Phi^n)}{\mathcal A\Phi^n}}
	\leq
	C\norm{\Phi^n}^{\frac32}
	\abs{\mathcal A\Phi^n}^{\frac32}
	\leq
	\frac14\abs{\mathcal A\Phi^n}^2+C_M .
	\]
	Similarly, by \eqref{equ:est R v in v} and Young's inequality,
	\[
	\abs{\inpro{\mathcal R(\Phi^n)}{\mathcal A\Phi^n}}
	\leq
	\frac14\abs{\mathcal A\Phi^n}^2+C_M .
	\]
	Moreover, by condition {\rm (G)},
	\[
	\sum_{k=1}^\infty \norm{P_n g_k(\Phi^n)}^2
	\leq
	C\left(1+\norm{\Phi^n}^2\right)
	\leq C_M .
	\]
	Consequently, for $t\leq\theta$, we obtain
	\[
	\mathcal{D}_n(t)\leq \abs{\mathcal A\Phi^n(t)}^2+C_M,
	\]
	where $C_M$ is independent of $n$, $\tau$, and $\delta$.
	
	Integrating \eqref{equ:ito d Phi n small time} from $0$ to $t\leq\theta$,
	we obtain
	\[
	\begin{aligned}
		\norm{\Phi^n(t)}^2+2I_n(t)
		&=
		\norm{\Phi_0^n}^2
		+
		\int_0^t \mathcal{D}_n(s)\,\mathrm ds
		+
		\mathcal M_n(t)
		\\
		&\leq
		\norm{\Phi_0^n}^2
		+
		I_n(t)
		+
		C_Mt
		+
		\mathcal M_n(t).
	\end{aligned}
	\]
	Hence, noting that $\norm{\Phi_0^n}^2\leq M$ a.s., we have
	\[
	\norm{\Phi^n(t)}^2+I_n(t)
	\leq
	M+C_M\delta+\mathcal M_n(t),
	\qquad \forall t\leq\theta .
	\]
	Therefore
	\begin{align}\label{equ:running-energy-bound}
		\sup_{t\in[0,\theta]}
		\left(
		\norm{\Phi^n(t)}^2+I_n(t)
		\right)
		\leq
		M+C_M\delta
		+
		\sup_{t\in[0,\theta]}\abs{\mathcal M_n(t)} .
	\end{align}
	
	Next, observe that
	\[
	\sup_{t\in[0,\theta]}\norm{\Phi^n(t)}^2
	+
	I_n(\theta)
	\leq
	2\sup_{t\in[0,\theta]}
	\left(
	\norm{\Phi^n(t)}^2+I_n(t)
	\right).
	\]
	Thus, by \eqref{equ:running-energy-bound}, we infer that
	\[
	A_n(\tau,\delta)
	\subset
	\left\{
	M+C_M\delta
	+
	\sup_{t\in[0,\theta]}\abs{\mathcal M_n(t)}
	>
	\frac{3M}{2}
	\right\}.
	\]
	Choose $\delta_0>0$ sufficiently small so that $C_M\delta_0\leq \frac{M}{4}$.
	Then, for $\delta \in (0,\delta_0]$, we have
	\begin{align}\label{equ:An Mn delta}
	A_n(\tau,\delta)
	\subset
	\left\{
	\sup_{t\in[0,\theta]}\abs{\mathcal M_n(t)}
	>
	\frac{M}{4}
	\right\}.
	\end{align}
	
	It remains to estimate the martingale term $\mathcal{M}_n$. Its quadratic variation satisfies
	\begin{align*}
		\langle \mathcal M_n\rangle_\theta
		&=
		4\int_0^\theta
		\sum_{k=1}^\infty
		\abs{
			\inpro{g_k(\Phi^n(s))}{\mathcal A\Phi^n(s)}
		}^2
		\ds,
	\end{align*}
	where we recall that $\theta=\tau\wedge\delta$.
	By condition {\rm (G)} and the stopped bound $\norm{\Phi^n(s)}^2\leq 4M$, we obtain for all $s\leq\theta$,
	\[
	\sum_{k=1}^\infty
	\abs{
		\inpro{g_k(\Phi^n(s))}{\mathcal A\Phi^n(s)}
	}^2
	\leq
	C\norm{\Phi^n(s)}^2
	\sum_{k=1}^\infty \norm{g_k(\Phi^n(s))}^2
	\leq C_M.
	\]
	This implies $\langle \mathcal M_n\rangle_\theta
	\leq C_M\delta$.
	Thus, by the Burkholder--Davis--Gundy inequality and Chebyshev's inequality,
	\[
	\begin{aligned}
		\bb P\left[
		\sup_{t\in[0,\theta]}
		\abs{\mathcal M_n(t)}
		>
		\frac{M}{4}
		\right]
		\leq
		\frac{16}{M^2}
		\bb E\left[
		\sup_{t\in[0,\theta]}
		\abs{\mathcal M_n(t)}^2
		\right]
		\leq
		\frac{C}{M^2}
		\bb E\left[
		\langle\mathcal M_n\rangle_\theta
		\right]
		\leq
		C_M\delta .
	\end{aligned}
	\]
	Hence, for all $\delta\in (0,\delta_0]$, we obtain from \eqref{equ:An Mn delta},
	\[
	\bb P\left[A_n(\tau,\delta)\right]
	\leq C_M\delta,
	\]
	where $C_M$ is independent of $n$, $\tau$, and $\delta$. Taking the
	supremum over $n\in\bb N$ and $\tau\in\mathcal T_n^{M,T}$ yields
	\[
	\sup_{n\in\bb N}
	\sup_{\tau\in\mathcal T_n^{M,T}}
	\bb P\left[A_n(\tau,\delta)\right]
	\leq C_M\delta .
	\]
	Letting $\delta\to 0^+$ proves \eqref{equ:lim sup prob}.
\end{proof}

We now prove the existence of a local strong pathwise solution to~\eqref{equ:equation}.

\begin{theorem}\label{the:strong sol}
Let $\Phi_0\in L^\infty(\Omega;\mathcal{V})$ be a given initial data. Then there exists a local strong solution $(\Phi,\tau)$ to the problem~\eqref{equ:equation} in the sense of Definition~\eqref{def:strong sol}.
\end{theorem}

\begin{proof}
We proceed with a compactness argument by applying Theorem~\ref{the:compactness} with $X=\mathcal{V}$ and $Y=\mathrm{D}(\mathcal{A})$, noting the definition of $\mathcal{T}_n^{M,T}$ in \eqref{equ:Tn MT}. Since the hypotheses of this theorem have been shown to hold in Proposition~\ref{pro:hypotheses}, we infer the existence of a subsequence of the Galerkin solutions $\{\Phi^{n_\ell}\}_{\ell\in \bb{N}}$ to \eqref{equ:galerkin}, a strictly positive stopping time $\tau\leq T$, and a process $\Phi(\cdot)= \Phi(\,\cdot\wedge\tau)$ which is continuous in $\mathcal{V}$, such that
\begin{align}\label{equ:conv as}
	\lim_{\ell\to\infty} \left(\sup_{t\in [0,\tau]} \norm{\Phi^{n_\ell}(t)-\Phi(t)}^2 + \int_0^\tau \abs{\mathcal{A} \big(\Phi^{n_\ell}(t) -\Phi(t)\big)}^2 \dt \right) = 0\; \text{ a.s.}
\end{align}
Furthermore, \eqref{equ:cauchy 4} implies for any $p\in [1,\infty)$,
\begin{align}\label{equ:reg Phi}
	\Phi(\,\cdot\wedge\tau) \in L^p\big(\Omega; C([0,T]; \mathcal{V})\big)\;
	\text{ and }\;
	\Phi \one_{t\leq \tau} \in L^p \big(\Omega; L^2(0,T;\mathrm{D}(\mathcal{A}))\big),
\end{align}
while \eqref{equ:cauchy 5} yields the existence of a collection of measurable sets $\{\Omega_\ell\}_{\ell\in \bb{N}}$ with 
\begin{align}\label{equ:bdd Lp Omega}
	\sup_{\ell\in \bb{N}} \bb{E} \left[ \one_{\Omega_\ell} \left( \sup_{t\in [0,\tau]} \norm{\Phi^{n_\ell}(t)}^2 + \int_0^\tau \abs{\mathcal{A}\Phi^{n_\ell}(t)}^2 \dt \right)^{\frac{p}{2}} \right] < \infty.
\end{align}
Thus, given \eqref{equ:conv as} and \eqref{equ:bdd Lp Omega}, by Lemma~\ref{lem:weak conv} we have
\begin{align}
	\label{equ:one Phi conv}
	&\one_{t\leq \tau} \one_{\Omega_\ell} \Phi^{n_\ell} \rightharpoonup \one_{t\leq \tau} \Phi \quad \text{in } L^p\big(\Omega; L^2(0,T;\mathrm{D}(\mathcal{A}))\big),
	\\
	\label{equ:one wedge tau conv}
	&\one_{\Omega_\ell} \Phi^{n_\ell}(\,\cdot\wedge\tau) \overset{\ast}{\rightharpoonup} \Phi \quad \text{in } L^p\big(\Omega;L^\infty(0,T;\mathcal{V})\big).
\end{align}
Next, we consider the convergence of the first nonlinear term $P_{n_\ell}\mathcal{B}(\Phi^{n_\ell})$. For any $v\in \mathcal{V}$, we have
\begin{align*}
	\abs{\inpro{P_{n_\ell} \mathcal{B}(\Phi^{n_\ell})- \mathcal{B}(\Phi)}{v}}
	&\leq
	\abs{\inpro{\mathcal{B}(\Phi^{n_\ell}-\Phi, \Phi^{n_\ell})}{P_{n_\ell} v}}
	+
	\abs{\inpro{\mathcal{B}(\Phi, \Phi^{n_\ell}-\Phi)}{P_{n_\ell} v}}
	+
	\abs{\inpro{\mathcal{B}(\Phi)}{Q_{n_\ell} v}}
	\\
	&\leq
	C \norm{\Phi^{n_\ell}-\Phi} \norm{\Phi^{n_\ell}} \norm{v}
	+
	C \norm{\Phi} \norm{\Phi^{n_\ell}-\Phi} \norm{v}
	+
	C \norm{\Phi}^2 \abs{Q_{n_\ell} v}^{\frac12} \norm{Q_{n_\ell} v}^{\frac12}
	\\
	&\leq
	C \norm{\Phi^{n_\ell}-\Phi} \norm{v} \left(\norm{\Phi^{n_\ell}}+ \norm{\Phi}\right) + C\lambda_{n_\ell}^{-\frac14} \norm{\Phi}^2 \norm{v},
\end{align*}
where in the penultimate step we used \eqref{equ:est B 3d alt}, while in the last step we applied \eqref{equ:poincare Qn}. This estimate, together with \eqref{equ:conv as}, implies that given any $v\in\mathcal{V}$, for a.e. $(\omega,t)\in \Omega\times [0,T]$,
\begin{align}\label{equ:1 P Bv conv}
	\one_{t\leq \tau} \inpro{P_{n_\ell}\mathcal{B}(\Phi^{n_\ell})}{v} \to \one_{t\leq \tau} \inpro{\mathcal{B}(\Phi)}{v}\; \text{ as $\ell\to\infty$}.
\end{align}
Furthermore, we have
\begin{align}\label{equ:sup int P bdd}
	\sup_{\ell\in \bb{N}} \bb{E}\left[\one_{\Omega_{\ell}} \int_0^\tau \abs{P_{n_\ell}\mathcal{B}(\Phi^{n_\ell})}^2 \dt \right]
	&\leq
	C\sup_{\ell\in \bb{N}} \bb{E} \left[\one_{\Omega_\ell} \int_0^\tau \norm{\Phi^{n_\ell}}^3 \abs{\mathcal{A}\Phi^{n_\ell}} \dt \right]
	\nonumber\\
	&\leq
	C \sup_{\ell\in \bb{N}} \bb{E} \left[\one_{\Omega_\ell} \left(\sup_{s\in [0,\tau]} \norm{\Phi^{n_\ell}(s)}^2 \right) \left(\int_0^\tau \abs{\mathcal{A}\Phi^{n_\ell}}^2 \dt\right)\right]
	\nonumber\\
	&\leq
	C \sup_{\ell\in \bb{N}} \bb{E} \left[\one_{\Omega_\ell} \left(\sup_{s\in [0,\tau]} \norm{\Phi^{n_\ell}(s)}^4 + \left(\int_0^\tau \abs{\mathcal{A}\Phi^{n_\ell}}^2 \dt \right)^2 \right)\right] < \infty,
\end{align}
where we applied Young's inequality and \eqref{equ:bdd Lp Omega} in the last line. Altogether, \eqref{equ:1 P Bv conv}, \eqref{equ:sup int P bdd}, and Lemma~\ref{lem:weak conv} imply that
\begin{align}\label{equ:one PB conv}
	\one_{t\leq \tau} \one_{\Omega_\ell}  P_{n_\ell} \mathcal{B}(\Phi^{n_\ell}) \rightharpoonup \one_{t\leq \tau} \mathcal{B}(\Phi) \quad \text{in } L^2\big(\Omega;L^2(0,T;\mathcal{H})\big).
\end{align}
Now, we study the convergence of the second nonlinear term $P_{n_\ell} \mathcal{R}(\Phi^{n_\ell})$. For any $v\in \mathcal{V}$, by the triangle and the H\"older inequalities we have
\begin{align*}
	&\abs{\inpro{P_{n_\ell} \mathcal{R}(\Phi^{n_\ell})- \mathcal{R}(\Phi)}{v}}
	\\
	&\leq
	\abs{\inpro{\mathcal{R}(\Phi^{n_\ell})- \mathcal{R}(\Phi)}{P_{n_\ell} v}}
	+
	\abs{\inpro{\mathcal{R}(\Phi)}{Q_{n_\ell} v}}
	\\
	&\leq
	C\left(1+\norm{\Phi^{n_\ell}}^{r-1} + \norm{\Phi}^{r-1}\right) \norm{\Phi^{n_\ell}-\Phi} \abs{v}
	+
	C \left(1+\norm{\Phi}^{r-1}\right) \norm{\Phi} \abs{Q_{n_\ell} v}
	\\
	&\leq
	C\left(1+\norm{\Phi^{n_\ell}}^{r-1} + \norm{\Phi}^{r-1}\right) \norm{\Phi^{n_\ell}-\Phi} \norm{v}
	+
	C\lambda_{n_\ell}^{-\frac12} \left(1+\norm{\Phi}^{r-1}\right) \norm{\Phi} \norm{v}.
\end{align*}
Here, in the penultimate step we employed \eqref{equ:est Rv1 minus v2} and \eqref{equ:est R v in v}, while in the last step we applied \eqref{equ:poincare Qn}. Similarly to \eqref{equ:1 P Bv conv}, the above estimate, together with \eqref{equ:conv as}, implies that for a.e. $(\omega,t)\in \Omega\times [0,T]$,
\begin{align}\label{equ:PR conv}
	\one_{t\leq \tau} \inpro{P_{n_\ell} \mathcal{R}(\Phi^{n_\ell})}{v} \to 
	\one_{t\leq \tau} \inpro{\mathcal{R}(\Phi)}{v}
	\; \text{ as $\ell\to\infty$}.
\end{align}
Moreover, noting \eqref{equ:bdd Lp Omega}, we have for some $r\geq 1$,
\begin{align}\label{equ:sup int PR}
	\sup_{\ell\in \bb{N}} \bb{E}\left[\one_{\Omega_{\ell}} \int_0^\tau \abs{P_{n_\ell}\mathcal{R}(\Phi^{n_\ell})}^2 \dt \right]
	&\leq
	C\sup_{\ell\in \bb{N}} \bb{E} \left[\one_{\Omega_\ell} \int_0^\tau \left(1+ \norm{\Phi^{n_\ell}}^{2(r-1)} \right) \norm{\Phi^{n_\ell}}^2 \dt \right]
	\nonumber\\
	&\leq
	C \sup_{\ell\in \bb{N}} \bb{E}\left[\one_{\Omega_\ell} \left( \sup_{s\in [0,\tau]} \norm{\Phi^{n_\ell}(s)}^2 + \norm{\Phi^{n_\ell}(s)}^{2r} \right) \right]
	<
	\infty.
\end{align}
Thus, \eqref{equ:PR conv}, \eqref{equ:sup int PR}, and Lemma~\ref{lem:weak conv} imply that
\begin{align}\label{equ:one PR conv}
	 \one_{t\leq \tau} \one_{\Omega_\ell} P_{n_\ell} \mathcal{R}(\Phi^{n_\ell}) \rightharpoonup \one_{t\leq \tau} \mathcal{R}(\Phi) \quad \text{in } L^2\big(\Omega;L^2(0,T;\mathcal{H})\big).
\end{align}
Finally, we consider the stochastic term. Noting the assumptions on $g$, we obtain
\begin{align*}
	\left(\sum_{k=1}^\infty \abs{P_{n_\ell} g_k(\Phi^{n_\ell})- g_k(\Phi)}^2 \right)^{\frac12}
	&\leq
	\left(\sum_{k=1}^\infty \abs{P_{n_\ell} g_k(\Phi^{n_\ell})- P_{n_\ell} g_k(\Phi)}^2 \right)^{\frac12} + \left(\sum_{k=1}^\infty \abs{Q_{n_\ell} g_k(\Phi)}^2 \right)^{\frac12}
	\\
	&\leq
	C\left( \norm{\Phi^{n_\ell}- \Phi} + \lambda_{n_\ell}^{-1} (1+\norm{\Phi})\right).
\end{align*}
This implies for a.e. $(\omega,t)\in \Omega\times [0,T]$, we have $\one_{t\leq \tau} P_{n_\ell} g(\Phi^{n_\ell}) \to \one_{t\leq \tau} g(\Phi)$ in $\ell^2(\mathcal{H})$. Furthermore, again by the assumptions on $g$ and \eqref{equ:bdd Lp Omega},
\begin{align*}
	\sup_{\ell\in \bb{N}} \bb{E}\left[\one_{\Omega_\ell} \int_0^\tau \sum_{k=1}^\infty \abs{P_{n_\ell} g_k(\Phi^{n_\ell})}^2 \dt \right]
	\leq
	C + C \sup_{\ell\in \bb{N}} \bb{E}\left[\one_{\Omega_\ell} \int_0^\tau \norm{\Phi^{n_\ell}}^2 \dt\right] < \infty,
\end{align*}
from which we infer that
\begin{align}\label{equ:Pg conv}
	\one_{t\leq \tau} \one_{\Omega_\ell} P_{n_\ell} g(\Phi^{n_\ell}) \rightharpoonup \one_{t\leq \tau} g(\Phi) \quad \text{in } L^2\big(\Omega; L^2(0,T; \ell^2(\mathcal{H})) \big).
\end{align}
Thus, from \eqref{equ:one Phi conv}, \eqref{equ:one PB conv}, \eqref{equ:one PR conv}, and \eqref{equ:Pg conv}, we deduce for any fixed $v\in \mathcal{H}$ the following weak convergences in $L^2\big(\Omega; L^2(0,T)\big)$:
\begin{equation}\label{equ:weak conv L2}
	\begin{aligned}
		\one_{\Omega_\ell} \int_0^{t\wedge \tau} \inpro{\mathcal{A}\Phi^{n_\ell}}{v} \ds &\rightharpoonup \int_0^{t\wedge\tau} \inpro{\mathcal{A} \Phi}{v} \ds,
		\\
		\one_{\Omega_\ell} \int_0^{t\wedge\tau} \inpro{P_{n_\ell} \mathcal{B}(\Phi^{n_\ell})}{v} \ds &\rightharpoonup \int_0^{t\wedge\tau} \inpro{\mathcal{B}(\Phi)}{v} \ds,
		\\
		\one_{\Omega_\ell} \sum_{k=1}^\infty \int_0^{t\wedge\tau} \inpro{P_{n_\ell} g_k(\Phi^{n_\ell})}{v} \mathrm{d}\beta_k &\rightharpoonup \sum_{k=1}^\infty \int_0^{t\wedge\tau} \inpro{g_k(\Phi)}{v} \mathrm{d}\beta_k.
	\end{aligned}
\end{equation}

Finally, we apply the following standard argument. Let $K\subset \Omega\times [0,T]$ be a measurable set. Then by \eqref{equ:one wedge tau conv} and \eqref{equ:weak conv L2}, we have
\begin{align*}
	\bb{E}\left[\int_0^T \one_K \inpro{\Phi(t)}{v} \dt \right]
	&=
	\lim_{\ell\to\infty} \bb{E}\left[\int_0^T \inpro{\one_{\Omega_\ell} \Phi^{n_\ell}(t\wedge\tau)}{\one_K v} \dt \right]
	\\
	&=
	\lim_{\ell\to\infty} \bb{E} \left[\int_0^T \one_K \one_{\Omega_\ell} \inpro{P_{n_\ell} \Phi_0}{v} \dt \right]
	\\
	&\quad
	- 
	\lim_{\ell\to \infty} \bb{E}\left[ \int_0^T \one_K \one_{\Omega_\ell} \left(\int_0^{t\wedge\tau} \inpro{\mathcal{A}\Phi^{n_\ell}+ P_{n_\ell} \mathcal{B}(\Phi^{n_\ell})+ P_{n_\ell} \mathcal{R}(\Phi^{n_\ell})}{v} \ds \right) \dt\right]
	\\
	&\quad
	+
	\lim_{\ell\to\infty} \bb{E} \left[\int_0^T \one_K \one_{\Omega_\ell} \left(\sum_{k=1}^\infty \int_0^{t\wedge\tau} \inpro{P_{n_\ell} g_k(\Phi^{n_\ell})}{v} \mathrm{d}\beta_k\right) \dt \right]
	\\
	&=
	\bb{E} \left[\int_0^T  \one_K \left(\inpro{\Phi_0}{v} - \int_0^{t\wedge\tau} \inpro{\mathcal{A} \Phi+\mathcal{B}(\Phi) + \mathcal{R}(\Phi)}{v} \ds \right) \dt\right]
	\\
	&\quad
	+ 
	\bb{E}\left[\int_0^T \one_K \left( \sum_{k=1}^\infty \int_0^{t\wedge\tau} \inpro{g_k(\Phi)}{v} \mathrm{d}\beta_k \right) \dt \right].
\end{align*}
Since $v$ and $K$ are arbitrary, we infer that $\Phi$ satisfies \eqref{equ:int equal} as an equation in $\mathcal{H}$. The required regularity for a strong solution is conferred by \eqref{equ:reg Phi}. This completes the proof of the theorem.
\end{proof}

The following theorem establishes the pathwise uniqueness of local strong solution obtained previously.

\begin{theorem}\label{the:unique}
	Let $\tau>0$ be a stopping time such that $(\Phi^{(1)},\tau)$ and $(\Phi^{(2)},\tau)$ are local strong solutions to \eqref{equ:equation} with the same initial data $\Phi_0$. Then
	\begin{align}\label{equ:prob phi t}
		\bb{P} \left[\Phi^{(1)}(t\wedge\tau)= \Phi^{(2)}(t\wedge\tau), \; \forall t\in [0,\infty)\right] = 1.
	\end{align}
\end{theorem}

\begin{proof}
Let $\Psi:= \Phi^{(1)}-\Phi^{(2)}$. Then $\Psi$ satisfies
\begin{align*}
	\mathrm{d}\Psi +\left(\mathcal{A} \Psi + \mathcal{B}(\Phi^{(1)})- \mathcal{B}(\Phi^{(2)}) + \mathcal{R}(\Phi^{(1)}) - \mathcal{R}(\Phi^{(2)}) \right) \dt = \sum_{k=1}^\infty \left(g_k(\Phi^{(1)})- g_k(\Phi^{(2)})\right) \mathrm{d}\beta_k.
\end{align*}
Applying It\^o's lemma to $\abs{\Psi}^2$ yields
\begin{align*}
	\mathrm{d}\abs{\Psi}^2
	&=
	-2 \norm{\Psi}^2 \dt
	-2 \inpro{\mathcal{B}(\Phi^{(1)})- \mathcal{B}(\Phi^{(2)})}{\Psi} \dt
	-2 \inpro{\mathcal{R}(\Phi^{(1)}) - \mathcal{R}(\Phi^{(2)})}{\Psi} \dt 
	\\
	&\quad
	+
	\sum_{k=1}^\infty \abs{g_k(\Phi^{(1)})- g_k(\Phi^{(2)})}^2 \dt 
	+
	2\sum_{k=1}^\infty \inpro{g_k(\Phi^{(1)})-g_k(\Phi^{(2)})}{\Psi} \mathrm{d}\beta_k.
\end{align*}
For any stopping time $\sigma\leq\tau$, we integrate the above expression over $[0,\sigma]$ and apply the expected value to obtain
\begin{align}\label{equ:e Psi sigma}
	&\bb{E} \left[\abs{\Psi(\sigma)}^2\right] + 2\bb{E} \left[\int_0^\sigma \norm{\Psi(t)}^2 \dt\right] 
	\nonumber\\
	&=
	-2 \bb{E}\left[\int_0^\sigma \inpro{\mathcal{B}(\Phi^{(1)})- \mathcal{B}(\Phi^{(2)})}{\Psi} \dt \right] 
	-
	2 \bb{E}\left[\int_0^\sigma \inpro{\mathcal{R}(\Phi^{(1)})- \mathcal{R}(\Phi^{(2)})}{\Psi} \dt \right] 
	\nonumber\\
	&\quad
	+
	\bb{E} \left[ \int_0^\sigma \abs{g_k(\Phi^{(1)})- g_k(\Phi^{(2)})}^2 \dt \right],
\end{align}
where we also used the martingale property of stochastic integral. Note that we have the following estimates. Firstly, by \eqref{equ:B antisym} and \eqref{equ:est B 3d alt},
\begin{align}\label{equ:B1 min B2}
	\abs{\inpro{\mathcal{B}(\Phi^{(1)})- \mathcal{B}(\Phi^{(2)})}{\Psi}}
	=
	\abs{\inpro{\mathcal{B}(\Psi,\Phi^{(1)})}{\Psi}}
	&\leq
	C\norm{\Psi}^{\frac32} \norm{\Phi^{(1)}} \abs{\Psi}^{\frac12}
	\nonumber\\
	&\leq
	\frac14 \norm{\Psi}^2 + C \norm{\Phi^{(1)}}^4 \abs{\Psi}^2,
\end{align}
where in the last step we used Young's inequality. Furthermore, by \eqref{equ:est Rv1 minus v2}, we have for some $r\geq 1$,
\begin{align}\label{equ:R1 min R2}
	\abs{\inpro{\mathcal{R}(\Phi^{(1)})- \mathcal{R}(\Phi^{(2)})}{\Psi}}
	&=
	C\left(1+\norm{\Phi^{(1)}}^{r-1} + \norm{\Phi^{(2)}}^{r-1} \right) \norm{\Psi} \abs{\Psi}
	\nonumber\\
	&\leq
	\frac14 \norm{\Psi}^2 + C\left(1+\norm{\Phi^{(1)}}^{2(r-1)} + \norm{\Phi^{(2)}}^{2(r-1)} \right) \abs{\Psi}^2.
\end{align}
Now, for any $R>0$, define the stopping time
\begin{align*}
	\sigma_R:= \inf\left\{t>0: \norm{\Phi^{(1)}(t)}^2 + \norm{\Phi^{(2)}(t)}^2 >R\right\} \wedge \tau.
\end{align*}
Noting this stopping time and Lipschitz assumptions on $g$, we employ estimates \eqref{equ:B1 min B2} and \eqref{equ:R1 min R2} in \eqref{equ:e Psi sigma} and rearrange the terms to obtain
\begin{align*}
	\bb{E} \left[\abs{\Psi(\sigma_R \wedge t)}^2\right] + \bb{E} \left[\int_0^{\sigma_R\wedge t} \norm{\Psi(s)}^2 \ds \right] 
	&\leq
	C\bb{E} \left[\int_0^{\sigma_R\wedge t} \left(1+\norm{\Phi^{(1)}}^{2(r-1)} + \norm{\Phi^{(2)}}^{2(r-1)} \right) \abs{\Psi (s)}^2 \ds \right]
	\\
	&\leq
	C(1+R^{r-1}) \,\bb{E} \left[\int_0^{\sigma_R\wedge t} \abs{\Psi (s)}^2 \ds \right].
\end{align*}
By the Gronwall lemma, we infer that
\begin{align*}
	\bb{E}\left[\abs{\Psi(\sigma_R \wedge t)}^2\right]=0,
\end{align*}
which implies $\Phi^{(1)}(\sigma_R\wedge t)= \Phi^{(2)}(\sigma_R\wedge t)$ a.s. Thus,
\begin{align*}
	\bb{P} \left[\Phi^{(1)}(t\wedge\tau) \neq \Phi^{(2)}(t\wedge\tau)\right]
	&\leq
	\bb{P} \left[ \{\sigma_R <\tau\} \cap \{\Phi^{(1)}(t\wedge\tau) \neq \Phi^{(2)}(t\wedge\tau)\}\right]
	\\
	&\leq
	\bb{P} \left[ \sigma_R<\tau \right] 
	\\
	&\leq
	\bb{P} \left[\sup_{s\in [0,\tau]} \left(\norm{\Phi^{(1)}(s)}^2 + \norm{\Phi^{(2)}(s)}^2\right) >R \right]
	\\
	&\leq
	\frac{1}{R} \bb{E}\left[ \sup_{s\in [0,\tau]} \left(\norm{\Phi^{(1)}(s)}^2 + \norm{\Phi^{(2)}(s)}^2\right) \right],
\end{align*}
which tends to $0$ as $R\to \infty$. Therefore for any $t\geq 0$, we have $\Phi^{(1)}(t\wedge\tau)= \Phi^{(2)}(t\wedge\tau)$ on a set of full measure which may depend on $t$. By taking the intersection of such sets corresponding to positive rational times, we obtain
\begin{align*}
	\bb{P} \left[\Phi^{(1)}(t\wedge\tau)= \Phi^{(2)}(t\wedge\tau), \; \forall t\in [0,\infty) \cap \bb{Q} \right] = 1.
\end{align*}
By the continuity in time property \eqref{equ:reg local strong} of the solution, we deduce \eqref{equ:prob phi t}, completing the proof of the theorem.
\end{proof}

\section{Existence of a maximal strong solution}\label{sec:max sol}

In this section, we establish the existence of a maximal strong solution $(\Phi, \xi)$ to \eqref{equ:equation}. We begin with the following estimate on a local weak solution of \eqref{equ:equation} which will be used on several occasions.

\begin{lemma}\label{lem:E abs Phi p}
	Suppose that $(\Phi,\tau)$ is a local weak solution of \eqref{equ:equation} in the sense of Definition~\ref{def:weak sol} corresponding to an initial data $\Phi_0\in L^\infty(\Omega;\mathcal{H})$. Then for any $T>0$ and $p\geq 4$,
	\begin{align*}
		\bb{E}\left[\sup_{t\in [0,\tau\wedge T]} \abs{\Phi(t)}^p \right]
		+
		\bb{E} \left[ \int_0^{\tau \wedge T} \norm{\Phi (s)}^2 \abs{\Phi (s)}^{p-2} \ds \right] 
		\leq
		C_p \left(1+\abs{\Phi_0}^p \right),
	\end{align*}
	where $C_p$ is a constant which depends on $p$ and $T$, but is independent of $\tau$.
\end{lemma}

\begin{proof}
For any stopping times $\sigma_a,\sigma_b$ such that $0\leq \sigma_a \leq t\leq \sigma_b \leq \tau\wedge T$, by applying It\^o's lemma to $\abs{\Phi}^p$, where $p\geq 4$, we have 
\begin{align*}
	&\abs{\Phi(t)}^p 
	+
	p \int_{\sigma_a}^{t} \norm{\Phi(s)}^2 \abs{\Phi(s)}^{p-2} \ds
	+
	p\int_{\sigma_a}^{t} \abs{\Phi(s)}^{p-2} \inpro{\mathcal{R}(\Phi(s))}{\Phi(s)} \ds
	\\
	&=
	\abs{\Phi(\sigma_a)}^p
	+
	\frac{p}{2} \int_{\sigma_a}^{t} \sum_{k=1}^\infty \abs{\Phi(s)}^{p-2} \abs{g_k(\Phi(s))}^2 \ds
	+
	\frac{p(p-2)}{2} \int_{\sigma_a}^t \sum_{k=1}^\infty \abs{\inpro{g_k(\Phi(s))}{\Phi(s)}}^2 \abs{\Phi(s)}^{p-4} \ds
	\\
	&\quad
	+
	p \int_{\sigma_a}^{t} \sum_{k=1}^\infty \abs{\Phi(s)}^{p-2} \inpro{g_k(\Phi(s))}{\Phi(s)} \mathrm{d}\beta_k.
\end{align*}
Noting assumption \eqref{equ:Rv v pos} on $\mathcal{R}$, we then have
\begin{align*}
	&\abs{\Phi(t)}^p 
	+
	p \int_{\sigma_a}^{t} \norm{\Phi(s)}^2 \abs{\Phi(s)}^{p-2} \ds
	\\
	&\leq
	\abs{\Phi(\sigma_a)}^p
	+
	p \int_{\sigma_a}^{t} \abs{\Phi(s)}^{p-2} \abs{\inpro{\mathcal{F}(\Phi(s))}{\Phi(s)}} \ds
	+
	\frac{p(p-1)}{2} \int_{\sigma_a}^{t} \sum_{k=1}^\infty \abs{\Phi(s)}^{p-2} \abs{g_k(\Phi(s))}^2 \ds
	\\
	&\quad
	+
	p \int_{\sigma_a}^{t} \sum_{k=1}^\infty \abs{\Phi(s)}^{p-2} \inpro{g_k(\Phi(s))}{\Phi(s)} \mathrm{d}\beta_k.
\end{align*}
Using the triangle inequality, taking supremum over $t\in [\sigma_a,\sigma_b]$, and applying the expected value, we obtain
\begin{align}\label{equ:E sup I1 to I3}
	&\bb{E} \left[\sup_{t\in [\sigma_a,\sigma_b]} \abs{\Phi(t)}^p\right] 
	+
	p \bb{E} \left[\int_{\sigma_a}^{\sigma_b} \norm{\Phi(s)}^2 \abs{\Phi(s)}^{p-2} \ds \right]
	\nonumber\\
	&\leq
	C_p \bb{E}\left[\abs{\Phi(\sigma_a)}^p \right] 
	+
	C_p \bb{E} \left[ \int_{\sigma_a}^{\sigma_b} \abs{\Phi(s)}^{p-2} \abs{\inpro{\mathcal{F}(\Phi(s))}{\Phi(s)}} \ds \right]
	\nonumber\\
	&\quad
	+
	C_p \bb{E}\left[  \int_{\sigma_a}^{\sigma_b} \sum_{k=1}^\infty \abs{\Phi(s)}^{p-2} \abs{g_k(\Phi(s))}^2 \ds \right] 
	\nonumber\\
	&\quad
	+
	C_p \bb{E}\left[ \sup_{t\in [\sigma_a,\sigma_b]} \abs{\int_{\sigma_a}^t \sum_{k=1}^\infty  \abs{\Phi(s)}^{p-2} \inpro{g_k(\Phi(s))}{\Phi(s)} \mathrm{d}\beta_k} \right]
	\nonumber\\
	&=: C_p \bb{E}\left[\abs{\Phi(\sigma_a)}^p \right]  + I_1+I_2+I_3.
\end{align}
We will estimate each term on the last line. Firstly, note that by \eqref{equ:Rv v pos}, we have
\begin{align}\label{equ:abs Phi p2}
	\abs{\Phi(s)}^{p-2} \abs{\inpro{\mathcal{F}(\Phi(s))}{\Phi(s)}}
	&\leq
	C \abs{\Phi(s)}^{p-2} \left(1+\norm{\Phi(s)}\right) \abs{\Phi(s)}
	\nonumber\\
	&\leq
	C \abs{\Phi(s)}^{p-1} + C \abs{\Phi(s)}^{\frac{p-2}{2}}  \norm{\Phi(s)} \abs{\Phi(s)}^{\frac{p}{2}}
	\nonumber\\
	&\leq
	C + C\abs{\Phi(s)}^p + \frac{p}{4} \norm{\Phi(s)}^2 \abs{\Phi(s)}^{p-2},
\end{align}
where in the last step we applied Young's inequality. Next, by the assumptions on $g$ and Young's inequality, we obtain
\begin{align}\label{equ:int sum Phi p2}
	\int_{\sigma_a}^{\sigma_b} \sum_{k=1}^\infty   \abs{\Phi(s)}^{p-2} \abs{g_k(\Phi(s))}^2 \ds
	&\leq
	C \int_{\sigma_a}^{\sigma_b} \abs{\Phi(s)}^{p-2} \left(1+\abs{\Phi(s)}^2\right) \ds
	\nonumber\\
	&\leq
	C\left(1+\int_{\sigma_a}^{\sigma_b} \abs{\Phi(s)}^p \ds \right).
\end{align}
Similarly, by the Burkholder--Davis--Gundy inequality, we also have
\begin{align}\label{equ:sup sum t}
	&\bb{E}\left[ \sup_{t\in [\sigma_a,\sigma_b]} \abs{\sum_{k=1}^\infty \int_{\sigma_a}^t \abs{\Phi(s)}^{p-2} \inpro{g_k(\Phi(s))}{\Phi(s)} \mathrm{d}\beta_k} \right]
	\nonumber\\
	&\leq
	C_p \bb{E} \left[\int_{\sigma_a}^{\sigma_b} \sum_{k=1}^\infty \abs{\inpro{g_k(\Phi(s))}{\Phi(s)}}^2 \abs{\Phi(s)}^{2(p-2)} \ds \right]^{\frac12}
	\nonumber\\
	&\leq
	C_p \bb{E} \left[\int_{\sigma_a}^{\sigma_b} \left(1+\abs{\Phi(s)}^2 \right) \abs{\Phi(s)}^{2(p-1)} \ds \right]^{\frac12}
	\nonumber\\
	&\leq
	\frac14 \bb{E} \left[\sup_{t\in [\sigma_a,\sigma_b]} \abs{\Phi(t)}^p\right]
	+
	C \bb{E} \left[1+\int_{\sigma_a}^{\sigma_b} \abs{\Phi(s)}^p \ds \right],
\end{align}
where in the last step we again used Young's inequality. Altogether, applying \eqref{equ:abs Phi p2}, \eqref{equ:int sum Phi p2}, and \eqref{equ:sup sum t} in \eqref{equ:E sup I1 to I3}, and rearranging the terms, we infer
\begin{align*}
	&\bb{E} \left[\sup_{t\in [\sigma_a,\sigma_b]} \abs{\Phi(t)}^p\right] 
	+
	p \bb{E} \left[\int_{\sigma_a}^{\sigma_b} \norm{\Phi(s)}^2 \abs{\Phi(s)}^{p-2} \ds \right]
	\le
	\bb{E}\left[1+\abs{\Phi(\sigma_a)}^p \right] 
	+
	C_p \bb{E} \left[\int_{\sigma_a}^{\sigma_b} \abs{\Phi(s)}^p \ds \right].
\end{align*}
The required inequality then follows from Lemma~\ref{lem:gronwall}.
\end{proof}

The existence of a maximal strong pathwise solution in the sense of Definition~\ref{def:global strong sol} will be established next. The argument follows closely the idea in~\cite[Section~5.3.4]{BreFeiHof18}

\begin{theorem}\label{the:max strong sol}
	There exists a unique maximal strong pathwise solution $(\Phi,\xi)$ to \eqref{equ:equation} and a sequence of stopping times $\{\tau_n\}_{n\in\bb{N}}$ announcing $\xi$. 
\end{theorem}

\begin{proof}
	With the aim of extending the solution to a maximal time of existence, we
	define the set
	\[
	\mathcal{L}:= \{\tau \text{ is a stopping time}: \exists \Phi
	\text{ such that } (\Phi,\tau) \text{ is a local strong pathwise solution} \}.
	\]
	The set $\mathcal{L}$ is non-empty by Theorem~\ref{the:strong sol}.
	
	We first note that $\mathcal L$ is upward directed. Indeed, if
	$\tau_1,\tau_2\in\mathcal L$, and if $(\Phi_1,\tau_1)$ and
	$(\Phi_2,\tau_2)$ are the corresponding local strong pathwise solutions,
	then by pathwise uniqueness, $\Phi_1(t)=\Phi_2(t)$ for all $t\in [0, \tau_1\wedge\tau_2]$ a.s.
	Thus, by pasting the two solutions, we obtain a local strong pathwise
	solution up to the stopping time $\tau_1\vee\tau_2$. Hence
	$\tau_1\vee\tau_2\in\mathcal L$.
	
	Define
	\[
	\xi:=\operatorname*{ess\,sup}_{\tau\in\mathcal L}\tau .
	\]
	Since $\mathcal L$ is upward directed, there exists an increasing sequence
	$\{\rho_n\}_{n\in\mathbb N}\subset\mathcal L$ such that $\rho_n\uparrow \xi$ a.s.
	For each $n\in\mathbb N$, let $\Phi_n$ denote a process such that
	$(\Phi_n,\rho_n)$ is a local strong pathwise solution. For $k,l\in\mathbb N$,
	define
	\[
	\Omega_{k,l}
	:=
	\left\{
	\Phi_k(t\wedge\rho_k\wedge\rho_l)
	=
	\Phi_l(t\wedge\rho_k\wedge\rho_l),
	\ \forall t\geq0
	\right\}.
	\]
	By pathwise uniqueness, $\mathbb P(\Omega_{k,l})=1$ for every $k,l$.
	Hence
	\[
	\widetilde\Omega:=\bigcap_{k,l\in\mathbb N}\Omega_{k,l}
	\]
	satisfies $\mathbb P(\widetilde\Omega)=1$.
	
	We now define the candidate maximal solution. For
	$\omega\in\widetilde\Omega$ and $t<\xi(\omega)$, choose
	$n=n(t,\omega)$ such that $t\leq\rho_n(\omega)$, and set
	\[
	\Phi(t,\omega):=\Phi_n(t,\omega).
	\]
	This definition is independent of the choice of $n$, again by pathwise
	uniqueness. On $\widetilde\Omega^c$, and for $t\geq\xi$, we define $\Phi$
	arbitrarily. Then, for every $n\in\mathbb N$, $(\Phi,\rho_n)$
	is a local strong pathwise solution. In particular, $\Phi$ is a strong pathwise solution on $[0,\xi)$, localised by the sequence $\{\rho_n\}_{n\in\mathbb N}$.
	
	Moreover, since each $(\Phi,\rho_n)$ is also a local weak solution, Lemma
	\ref{lem:E abs Phi p} gives, for every $T>0$ and $p\geq4$,
	\[
	\mathbb E\left[
	\sup_{t\in[0,\rho_n\wedge T]}|\Phi(t)|^p
	\right]
	+
	\mathbb E\left[
	\int_0^{\rho_n\wedge T}
	\norm{\Phi(s)}^2 |\Phi(s)|^{p-2}\,\mathrm ds
	\right]
	\leq
	C_p(1+|\Phi_0|^p),
	\]
	where $C_p$ is independent of $n$. Letting $n\to\infty$ and using the
	monotone convergence theorem, we obtain the corresponding bound on
	$[0,\xi\wedge T)$.
	
	We next prove maximality. Let $(\Psi,\sigma)$ be any other local strong
	pathwise solution with the same initial datum. Then $\sigma\in\mathcal L$,
	and hence, by the definition of $\xi$, we have $\sigma\leq \xi$ a.s.
	Furthermore, by pathwise uniqueness, $\Psi(t)=\Phi(t)$ for all $t\in [0,\sigma]$ a.s.
	Thus no local strong pathwise solution can extend $(\Phi,\xi)$ beyond
	$\xi$ on a set of positive probability.
	
It remains to construct an announcing sequence. We first claim that, for
every fixed $n\in\mathbb N$ and $T>0$,
\[
\mathbb P\big(\rho_n\wedge T=\xi\big)=0.
\]
Indeed, fix $n\in\mathbb N$ and $T>0$, and set $\theta_{n,T}:=\rho_n\wedge T$.
Since $(\Phi,\rho_n)$ is a local strong pathwise solution, the stopped pair
$(\Phi,\theta_{n,T})$ is also a local strong pathwise solution. We now
restart the equation at the stopping time $\theta_{n,T}$. For $t\geq 0$, consider the
shifted filtration $\widehat{\mathcal F}_t:=\mathcal F_{\theta_{n,T}+t}$,
and define the shifted Brownian motions
\[
\widehat\beta_k(t)
:=
\beta_k(\theta_{n,T}+t)-\beta_k(\theta_{n,T}).
\]
By the strong Markov property of Brownian motion, the processes
$\{\widehat\beta_k\}_{k\in\mathbb N}$ are independent Brownian motions with
respect to $\{\widehat{\mathcal F}_t\}_{t\geq0}$.
By the local existence theorem applied on this shifted stochastic basis
with initial datum $\Phi(\theta_{n,T})$, and using the usual localisation if
$\Phi(\theta_{n,T})$ is not essentially bounded in $\mathcal V$, there exists a local strong
pathwise solution $(\widehat\Phi,\widehat\tau)$ such that
\[
\widehat\Phi(0)=\Phi(\theta_{n,T}),
\qquad
\widehat\tau>0
\quad\text{a.s.}
\]
Define the pasted process
\[
\Phi_{\mathrm e}(t):=
\begin{cases}
	\Phi(t), & 0\leq t\leq \theta_{n,T},\\[1mm]
	\widehat\Phi(t-\theta_{n,T}),
	& \theta_{n,T}<t\leq \theta_{n,T}+\widehat\tau .
\end{cases}
\]
Since $\widehat\Phi(0)=\Phi(\theta_{n,T})$, the pasted process has the
required path continuity at the pasting time. Moreover, splitting the drift
and stochastic integrals at $\theta_{n,T}$ and using
\[
\widehat\beta_k(t)
=
\beta_k(\theta_{n,T}+t)-\beta_k(\theta_{n,T}),
\]
we see that $(\Phi_{\mathrm e},\theta_{n,T}+\widehat\tau)$ is a local
strong pathwise solution on the original stochastic basis.
Hence, $\theta_{n,T}+\widehat\tau\in\mathcal L$. By the
definition of $\xi$ as the essential supremum of $\mathcal{L}$, we have $\theta_{n,T}+\widehat\tau\leq \xi$ a.s.
However, on the event $\{\rho_n\wedge T=\xi\}$, we have
$\theta_{n,T}=\xi$, and thus
\[
\theta_{n,T}+\widehat\tau
=
\xi+\widehat\tau
>
\xi
\quad\text{a.s. on } \{\rho_n\wedge T=\xi\},
\]
which contradicts the preceding inequality unless $\mathbb P(\rho_n\wedge T=\xi)=0$.
	
Now, define
	\[
	\tau_n:=\rho_n\wedge n,
	\quad \forall n\in\mathbb N.
	\]
	Then $\{\tau_n\}_{n\in\mathbb N}$ is an increasing sequence of stopping
	times and $\tau_n\uparrow\xi$ a.s.
	Moreover, $\tau_n<\xi$ a.s. on $\{\xi>0\}$. Indeed, if $\xi=\infty$, then
	$\tau_n\leq n<\xi$. If $\xi<\infty$ and $\tau_n=\xi$, then $\rho_n\wedge n=\xi<\infty$,
	which has probability zero by the preceding paragraph with $T=n$. Hence
	$\{\tau_n\}_{n\in\mathbb N}$ announces $\xi$.
	
	Finally, the uniqueness of the maximal strong pathwise solution follows from
	pathwise uniqueness. Indeed, if $(\Psi,\zeta)$ is another maximal strong
	pathwise solution, then by the preceding maximality argument,
	$\zeta\leq\xi$ a.s. If $\mathbb P(\zeta<\xi)>0$, then $\Phi$ extends
	$\Psi$ beyond $\zeta$ on a set of positive probability, contradicting the
	maximality of $(\Psi,\zeta)$. Therefore $\zeta=\xi$ a.s., and
	$\Psi=\Phi$ up to $\xi$ by pathwise uniqueness. This completes the proof.
\end{proof}

\section{Existence of a global strong pathwise solution for $d\leq 2$}\label{sec:strong}

The existence of a global strong pathwise solution to \eqref{equ:equation} for $d\leq 2$ will be established in this section. We remind the reader that, in addition to conditions (B), (R), and (G) stated in Section~\ref{sec:assum}, we impose the growth condition \eqref{equ:ass global} in order to obtain global-in-time strong solutions.

\begin{theorem}\label{the:global strong 2d}
Let $d\leq 2$ and suppose that all the assumptions in Section~\ref{sec:assum}, including \eqref{equ:ass global}, hold. The maximal strong pathwise solution $(\Phi, \xi)$ of \eqref{equ:equation} corresponding to an initial data $\Phi_0\in L^\infty(\Omega;\mathcal{V})$ is global in the sense of Definition~\ref{def:global strong sol}. In particular, there exists an increasing sequence of strictly positive stopping times $\{\tau_n\}_{n\in\bb{N}}$ with $\tau_n\uparrow \xi$, with~$\bb{P}\left[\xi=\infty\right]=1$. 
\end{theorem}

\begin{proof}
Let $\{\tau_n\}_{n\in\bb{N}}$ be a strictly increasing sequence of stopping times announcing $\xi$. Note that
\begin{align}\label{equ:xi less inf}
	\{\xi<\infty\} = \bigcup_{T=1}^\infty \{\xi\leq T\} = \bigcup_{T=1}^\infty \bigcap_{n=1}^\infty \{\tau_n\leq T\}.
\end{align}
Since $\tau_n$ is increasing, we have
\begin{align}\label{equ:tau n less t}
	\bb{P} \left[\bigcap_{n=1}^\infty \{\tau_n\leq T\} \right] 
	=
	\lim_{n\to \infty} \bb{P}\left[\tau_n \leq T\right].
\end{align}
Therefore, to show that $\bb{P}\left[\xi< \infty\right]=0$, by \eqref{equ:xi less inf} and \eqref{equ:tau n less t}, it suffices to show that for any $T\geq 1$,
\begin{equation}\label{equ:lim n P tau n}
	\lim_{n\to \infty} \bb{P}\left[\tau_n\leq T\right] = 0.
\end{equation}
To this end, with $p$ given by \eqref{equ:ass global}, for any $M>0$ we define the stopping time
\begin{align}\label{equ:sigma M}
	\sigma_M:= \inf_{t\geq 0} \left\{\int_0^{t\wedge\xi} \abs{\Phi(s)}^{4p} \norm{\Phi(s)}^2 \ds > M\right\} \wedge 2T.
\end{align}
Note that
\begin{align}\label{equ:P tau n less t}
	\bb{P} \left[\tau_n\leq T\right]
	\leq
	\bb{P} \left[\sigma_M\leq T\right]
	+
	\bb{P} \left[ \left\{\tau_n\leq T\right\} \cap \left\{\sigma_M>T\right\}\right] 
	 =: I_1+I_2.
\end{align}
The term $I_1$ can be bounded by the Chebyshev inequality as follows:
\begin{align}\label{equ:I1 P}
	I_1
	&\leq
	\bb{P} \left[\int_0^{T\wedge\xi} \abs{\Phi(s)}^{4p} \norm{\Phi(s)}^2 \ds \geq M\right] 
	\leq
	\frac{1}{M} \bb{E} \left[\int_0^{T\wedge\xi} \abs{\Phi(s)}^{4p} \norm{\Phi(s)}^2 \ds \right].
\end{align}
Note that the last term tends to zero as $M\to\infty$ by Lemma~\ref{lem:E abs Phi p}. Next, by definition, for the term $I_2$ in \eqref{equ:P tau n less t}, we have
\begin{align}\label{equ:I2 P}
	I_2 &\leq 
	\bb{P} \left[ \left\{ \sup_{t\in [0,\tau_n\wedge T]} \norm{\Phi(t)}^2 + \int_0^{\tau_n \wedge T} \abs{\mathcal{A} \Phi(s)}^2 \ds \geq n \right\} 
	\cap \left\{\sigma_M>T\right\}\right] 
	\nonumber\\
	&\leq
	\bb{P} \left[ \sup_{t\in [0,\tau_n\wedge \sigma_M]} \norm{\Phi(t)}^2 + \int_0^{\tau_n \wedge \sigma_M} \abs{\mathcal{A} \Phi(s)}^2 \ds \geq n \right].
\end{align}
We need to estimate this last probability. Let $T,M$, and $n$ be fixed. Let $\tau_a,\tau_b$ be a pair of stopping times such that $0\leq \tau_a \leq \tau_b \leq \tau_n\wedge \sigma_M$. Applying It\^o's lemma to $\norm{\Phi}^2$ and following the same argument as in~\eqref{equ:I1a to I4a}, we have
\begin{align}\label{equ:sub tau J1 J4}
	&\bb{E} \left[ \sup_{t\in [\tau_a,\tau_b]} \norm{\Phi(t)}^2 \right] 
	+
	2 \bb{E} \left[ \int_{\tau_a}^{\tau_b} \abs{\mathcal{A} \Phi(s)}^2 \ds \right] 
	\nonumber\\
	&\leq
	\bb{E} \left[\norm{\Phi(\tau_a)}^2\right] 
	+
	C \bb{E}\left[\int_{\tau_a}^{\tau_b} \abs{\inpro{\mathcal{B}(\Phi)}{\mathcal{A}\Phi}}\ds \right]
	+
	C \bb{E} \left[\int_{\tau_a}^{\tau_b} \abs{\inpro{\mathcal{R}(\Phi)}{\mathcal{A}\Phi}}\ds \right]
	\nonumber\\
	&\quad
	+
	\bb{E} \left[\int_{\tau_a}^{\tau_b} \sum_{k=1}^\infty \norm{g_k(\Phi)}^2 \ds \right]
	+
	\bb{E} \left[\sup_{t\in [\tau_a, \tau_b]} \abs{\sum_{k=1}^\infty \int_{\tau_a}^t 2\inpro{g_k(\Phi)}{\Phi} \mathrm{d}\beta_k} \right].
	\nonumber\\
	&=: \bb{E} \left[\norm{\Phi(\tau_a)}^2\right] + J_1+J_2+J_3+J_4.
\end{align}
For the term $J_1$, we use \eqref{equ:est B with A} for $d=2$ and Young's inequality to obtain
\begin{align*}
	J_1
	&\leq
	C\bb{E} \left[\int_{\tau_a}^{\tau_b} \abs{\Phi}^{\frac12} \norm{\Phi} \abs{\mathcal{A} \Phi}^{\frac32} \ds \right] 
	\leq
	C \bb{E} \left[ \int_{\tau_a}^{\tau_b} \abs{\Phi}^2 \norm{\Phi}^4 \ds \right] 
	+
	\frac12 \bb{E} \left[\int_{\tau_a}^{\tau_b} \abs{\mathcal{A} \Phi}^2 \ds \right].
\end{align*}
Next, we apply \eqref{equ:ass global} and Young's inequality to estimate $J_2$ as
\begin{align*}
	J_2
	&\leq
	C\bb{E} \left[\int_{\tau_a}^{\tau_b} \left(1+\abs{\Phi}^p \abs{\mathcal{A} \Phi}^{\frac12}\right) \norm{\Phi} \abs{\mathcal{A} \Phi} \ds \right]
	\\
	&\leq
	C \bb{E} \left[ \int_{\tau_a}^{\tau_b} \norm{\Phi}^2+ \abs{\Phi}^{4p} \norm{\Phi}^4 \ds \right] 
	+
	\frac12 \bb{E} \left[\int_{\tau_a}^{\tau_b} \abs{\mathcal{A} \Phi}^2 \ds \right].
\end{align*}
By the assumption on $g$, we also have
\begin{align*}
	J_3 \leq
	C\bb{E}\left[\int_{\tau_a}^{\tau_b} \left(1+\norm{\Phi}^2\right) \ds\right].
\end{align*}
For the last term, by the Burkholder--Davis--Gundy inequality, we have
\begin{align*}
	J_4
	\leq
	\frac12 \bb{E} \left[\sup_{t\in [\tau_a, \tau_b]} \norm{\Phi(t)}^2 \right] 
	+
	C\bb{E} \left[\int_{\tau_a}^{\tau_b} \left(1+\norm{\Phi}^2\right) \ds \right].
\end{align*}
Substituting these estimates into \eqref{equ:sub tau J1 J4} and rearranging the terms, we obtain for some $s\geq 0$,
\begin{align}\label{equ:E sup norm Phi}
	&\bb{E} \left[ \sup_{t\in [\tau_a,\tau_b]} \norm{\Phi(t)}^2 \right] 
	+
	\bb{E} \left[ \int_{\tau_a}^{\tau_b} \abs{\mathcal{A} \Phi(s)}^2 \ds \right] 
	\nonumber\\
	&\leq
	C \bb{E} \left[\norm{\Phi(\tau_a)}^2\right] 
	+
	C \bb{E} \left[ \int_{\tau_a}^{\tau_b} \left(1+\abs{\Phi}^{4p} \norm{\Phi}^2 \right) \norm{\Phi}^2 \ds \right].
\end{align}
Note that by the definition of $\sigma_M$ in \eqref{equ:sigma M}, we have
\begin{align*}
	\int_0^{\sigma_M} \abs{\Phi(s)}^{4p} \norm{\Phi(s)}^2 \ds \leq M, \; \text{ a.s.}
\end{align*}
We apply the Gronwall lemma for stochastic processes (Lemma~\ref{lem:gronwall}) to \eqref{equ:E sup norm Phi} to obtain
\begin{align*}
	&\bb{E} \left[ \sup_{t\in [0,\tau_n\wedge\sigma_M]} \norm{\Phi(t)}^2 + \int_0^{\tau_n \wedge\sigma_M} \abs{\mathcal{A} \Phi(s)}^2 \ds \right] 
	\leq
	C_{T,M} \left(1+ \norm{\Phi_0}^2 \right),
\end{align*}
where $C_{T,M}$ depends on $T$ and $M$, but is independent of $N$. Now, continuing from \eqref{equ:I2 P}, by the Chebyshev inequality we have
\begin{align}\label{equ:I2 P fin}
	I_2
	&\leq
	\frac{1}{n}  \bb{E} \left[ \sup_{t\in [0,\tau_n\wedge\sigma_M]} \norm{\Phi(t)}^2 + \int_0^{\tau_n \wedge\sigma_M} \abs{\mathcal{A} \Phi(s)}^2 \ds \right]
	\leq
	\frac{C_{T,M}}{n} \left(1+ \norm{\Phi_0}^2 \right).
\end{align}
Altogether, from \eqref{equ:P tau n less t}, \eqref{equ:I1 P}, and \eqref{equ:I2 P fin}, we deduce that for any $T\geq 1$,
\begin{align*}
	\bb{P}\left[\tau_n \leq T\right] \leq \frac{1}{M} \bb{E} \left[\int_0^{T\wedge\xi} \abs{\Phi(s)}^{4p} \norm{\Phi(s)}^2 \ds \right]
	+
	\frac{C_{T,M}}{n} \left(1+ \norm{\Phi_0}^2 \right).
\end{align*}
First, we let $n\to\infty$ to obtain
\begin{align*}
	\lim_{n\to\infty} \bb{P}\left[\tau_n \leq T\right] \leq \frac{1}{M} \bb{E} \left[\int_0^{T\wedge\xi} \abs{\Phi(s)}^{4p} \norm{\Phi(s)}^2 \ds \right].
\end{align*}
Since this is true for any $M>0$, we conclude that $\lim_{n\to\infty} \bb{P}\left[\tau_n \leq T\right]=0$, thus $\bb{P}\left[\xi< \infty\right]=0$, as required in \eqref{equ:lim n P tau n}.
\end{proof}

\appendix

\section{Auxiliary results} \label{sec:app}

For convenience, we collect several auxiliary results that are used in the paper. The first is a compactness theorem for stochastic processes satisfying a certain Cauchy criterion, which we adapt from~\cite{GlaZia09, GlaTem11}. The second is a generalised Poincar\'e-type inequalities, while the third is a general lemma concerning weak convergence in Banach spaces. The last result is a Gronwall-type lemma for stochastic processes~\cite{GlaZia09}.

\begin{theorem}[\cite{GlaZia09, GlaTem11}]\label{the:compactness}
Let $T>0$ be given.
Suppose that $(X, \abs{\cdot}_X)$ and $(Y, \abs{\cdot}_Y)$ are Banach spaces with continuous embedding $Y\subset X$. Define the space
\[
\mathcal{E}(T):= C([0,T]; X) \cap L^2(0,T; Y)
\]
with the norm
\[
\abs{v}_{\mathcal{E}(T)}:= \left( \sup_{t\in [0,T]} \abs{v(t)}_X^2 + \int_0^T \abs{v(t)}_Y^2 \dt \right)^{\frac12}.
\]
Let $\{v_n\}_{n\in\bb{N}}$ be a sequence of $Y$-valued stochastic processes so that $v_n\in \mathcal{E}(T)$ a.s. Suppose that there exists a deterministic constant $M>1$ such that $\abs{v_n(0)}_X^2\leq M$ a.s. Define a collection of stopping times
\begin{align*}
	\mathcal{T}_{n}^{M,T}:= \left\{\tau\leq T: \abs{v_n}_{\mathcal{E}(\tau)}^2 \leq 4M \right\},
\end{align*}
and let $\mathcal{T}_{m,n}^{M,T}:= \mathcal{T}_m^{M,T} \cap \mathcal{T}_n^{M,T}$.

Suppose that we have
\begin{align}
	\label{equ:cauchy 1}
	\lim_{n\to \infty} \sup_{m\geq n} \,\sup_{\tau \in \mathcal{T}_{m,n}^{M,T}} \bb{E}\left[ \abs{v_n-v_m}_{\mathcal{E}(\tau)} \right] &= 0.
\end{align}
Assume further that there exists a number $\widetilde{M}\in (M,4M)$ such that
\begin{align}
	\label{equ:cauchy 2}
	\lim_{\delta\to 0^+} \sup_{n\in \bb{N}} \,\sup_{\tau \in \mathcal{T}_{n}^{M,T}}
	\bb{P} \left[ \abs{v_n}_{\mathcal{E} (\tau \wedge \delta)}^2 > \widetilde{M} \right] &= 0.
\end{align}
Then there exists a stopping time $\tau$ with $\bb{P}(0<\tau\leq T)=1$, a process $v$ with $v(\,\cdot\,\wedge\tau) \in \mathcal{E}(\tau)$, and a subsequence $\{v_{n_l}\}$, such that $\bb{P}$-a.s.,
\begin{align}\label{equ:cauchy 3}
	\lim_{l\to\infty} \abs{v_{n_l}-v}_{\mathcal{E}(\tau)} = 0.
\end{align}
Moreover, for every $p\in [1,\infty)$,
\begin{align}\label{equ:cauchy 4}
	\bb{E}\left[ \abs{v}_{\mathcal{E}(\tau)}^p \right] &\leq C_p M^p.
\end{align}
Furthermore, there exists a sequence of measurable sets $\Omega_l \uparrow \Omega$ such that
\begin{align}\label{equ:cauchy 5}
	\sup_{l\in \bb{N}} \bb{E}\left[\one_{\Omega_l} \abs{v_{n_l}}_{\mathcal{E}(\tau)}^p \right] &< \infty,
	\\
	\label{equ:lim cauchy 6}
	\lim_{l\to\infty} \bb{E}\left[\one_{\Omega_l} \abs{v_{n_l}-v}_{\mathcal{E}(\tau)}^p \right] &= 0.
\end{align}
\end{theorem}

\begin{proof}
	We argue as in the proof of the abstract comparison lemma of
	Glatt-Holtz--Ziane~\cite{GlaZia09}, with the exit levels modified to match the present assumptions.
	First, we choose numbers $a,b$ such that
	\[
	\sqrt{\widetilde M}<a<b<2\sqrt M .
	\]
	For $l\in\mathbb N$, set
	\[
	a_l:=a+(b-a)2^{-l},
	\qquad
	d_l:=a_l-a_{l+1},
	\qquad
	\varepsilon_l:=\frac{d_l}{3}.
	\]
	Then
	\[
	\sqrt{\widetilde M}<a<a_{l+1}<a_l<b<2\sqrt M .
	\]
	Thanks to \eqref{equ:cauchy 1}, we may choose a strictly increasing sequence
	$\{n_l\}_{l\in\mathbb N}$ such that
	\[
	\sup_{m\geq n_l}
	\sup_{\sigma\in\mathcal T_{m,n_l}^{M,T}}
	\mathbb E
	\left[
	\abs{v_{n_l}-v_m}_{\mathcal E(\sigma)}
	\right]
	\leq
	2^{-2l}\varepsilon_l .
	\]
	In particular,
	\[
	\sup_{\sigma \in \mathcal T_{n_{l+1},n_{l}}^{M,T}} 
	\mathbb E
	\left[
	\abs{v_{n_l}-v_{n_{l+1}}}_{\mathcal E(\sigma)}
	\right]
	\leq
	2^{-2l}\varepsilon_l.
	\]
	
	Define the stopping times
	\[
	\tau_l
	:=
	\inf\left\{
	t\geq 0:
	\abs{v_{n_l}}_{\mathcal E(t)}>a_l
	\right\}\wedge T .
	\]
	Since $\abs{v_{n_l}}_{\mathcal E(0)}
	=
	\abs{v_{n_l}(0)}_X
	\leq \sqrt M
	<
	a_l$, and since $t\mapsto \abs{v_{n_l}}_{\mathcal E(t)}$ is continuous, we have
	\[
	\abs{v_{n_l}}_{\mathcal E(\tau_l)}\leq a_l<2\sqrt M .
	\]
	Hence, $\abs{v_{n_l}}_{\mathcal E(\tau_l)}^2<4M$, so that $\tau_l\in\mathcal T_{n_l}^{M,T}$. Consequently, $\tau_{l} \wedge\tau_{l+1} \in \mathcal T_{n_{l+1},n_{l}}^{M,T}$.
	Now define
	\begin{equation}\label{equ:Omega N}
	\Omega_N
	:=
	\bigcap_{j=N}^{\infty}
	\left\{
	\abs{v_{n_j}-v_{n_{j+1}}}
	_{\mathcal E(\tau_j\wedge\tau_{j+1})}
	<
	\varepsilon_j
	\right\}.
	\end{equation}
	By Markov's inequality,
	\[
	\begin{aligned}
		\mathbb P\left[
		\abs{v_{n_l}-v_{n_{l+1}}}
		_{\mathcal E(\tau_l\wedge\tau_{l+1})}
		\geq \varepsilon_l
		\right]
		&\leq
		\varepsilon_l^{-1}
		\mathbb E
		\left[
		\abs{v_{n_l}-v_{n_{l+1}}}
		_{\mathcal E(\tau_l\wedge\tau_{l+1})}
		\right]
		\leq
		2^{-2l}.
	\end{aligned}
	\]
	Thus, by the Borel--Cantelli lemma, we infer that
	\[
	\mathbb P\left(\bigcup_{N=1}^{\infty}\Omega_N\right)=1.
	\]
	
	We next show that, on $\Omega_N$, the sequence $\{\tau_l\}_{l\geq N}$ is
	non-increasing. Indeed, fix $\omega\in\Omega_N$ and suppose
	$\tau_{l+1}(\omega)>\tau_l(\omega)$ for some $l\geq N$. Then
	$\tau_l(\omega)<T$, and by continuity, $\abs{v_{n_l}(\omega)}_{\mathcal E(\tau_l(\omega))}=a_l$.
	On the other hand, since $\omega\in\Omega_N$,
	\[
	\begin{aligned}
		\abs{v_{n_{l+1}}(\omega)}_{\mathcal E(\tau_l(\omega))}
		&\geq
		\abs{v_{n_l}(\omega)}_{\mathcal E(\tau_l(\omega))}
		-
		\abs{v_{n_l}(\omega)-v_{n_{l+1}}(\omega)}
		_{\mathcal E(\tau_l(\omega)\wedge\tau_{l+1}(\omega))}
		\\
		&>
		a_l-\varepsilon_l
		>
		a_{l+1}.
	\end{aligned}
	\]
	This contradicts the definition of $\tau_{l+1}$, since
	$\tau_l(\omega)<\tau_{l+1}(\omega)$. Hence
	\[
	\tau_{l+1}(\omega)\leq \tau_l(\omega),
	\quad \text{for every } l\geq N,\; \omega\in\Omega_N .
	\]
	Therefore the limit
	\[
	\tau:=\lim_{l\to\infty}\tau_l
	\]
	exists a.s. Since the $\tau_l$ are stopping times, $\tau$ is a stopping
	time, and clearly $\tau\leq T$ a.s.
	
	It remains to show that $\tau>0$ a.s. Let $0<\delta<T$. Since the
	sequence $\{\tau_l\}$ is eventually non-increasing a.s.,
	\[
	\mathbb P(\tau<\delta)
	\leq
	\limsup_{l\to\infty}\,\mathbb P(\tau_l<\delta).
	\]
	If $\tau_l<\delta$, then
	\[
	\abs{v_{n_l}}_{\mathcal E(\tau_l\wedge\delta)}
	=
	\abs{v_{n_l}}_{\mathcal E(\tau_l)}
	=
	a_l
	>
	\sqrt{\widetilde M}.
	\]
	Since $\tau_l\in\mathcal T_{n_l}^{M,T}$, it follows from
	\eqref{equ:cauchy 2} that
	\[
	\begin{aligned}
		\lim_{\delta\to0}
		\sup_l\mathbb P(\tau_l<\delta)
		&\leq
		\lim_{\delta\to0}
		\sup_l
		\mathbb P\left[
		\abs{v_{n_l}}_{\mathcal E(\tau_l\wedge\delta)}^2
		>
		\widetilde M
		\right]
		\\
		&\leq
		\lim_{\delta\to0}
		\sup_{n\in\mathbb N}
		\sup_{\sigma\in\mathcal T_n^{M,T}}
		\mathbb P\left[
		\abs{v_n}_{\mathcal E(\sigma\wedge\delta)}^2
		>
		\widetilde M
		\right]
		=0.
	\end{aligned}
	\]
	Thus $\mathbb P(\tau=0)=0$, and hence
	\[
	\mathbb P(0<\tau\leq T)=1.
	\]
	
	We now prove the convergence. Recall the set $\Omega_N$ defined in \eqref{equ:Omega N}, and that $\mathbb P\big(\bigcup_{N=1}^{\infty}\Omega_N\big)=1$. If $\omega\in\Omega_N$ and $l\geq N$, then
	$\tau(\omega)\leq \tau_{l+1}(\omega)\leq\tau_l(\omega)$, and therefore
	\[
		\abs{v_{n_l}(\omega)-v_{n_{l+1}}(\omega)}_{\mathcal E(\tau(\omega))}
		\leq
		\abs{v_{n_l}(\omega)-v_{n_{l+1}}(\omega)}
		_{\mathcal E(\tau_l(\omega)\wedge\tau_{l+1}(\omega))}
		<
		\varepsilon_l.
	\]
	Since $\sum_{l\in \bb{N}} \varepsilon_l<\infty$, the sequence $\{v_{n_l}\}$ is Cauchy
	in $\mathcal E(\tau)$ almost surely. Hence, there exists a process
	$v=v(\cdot\wedge\tau)\in\mathcal E(\tau)$ such that $\abs{v_{n_l}-v}_{\mathcal E(\tau)}\to 0$ a.s., showing \eqref{equ:cauchy 3}.
	
	Finally, on $\Omega_N$, for every $l\geq N$,
	\[
	\abs{v_{n_l}}_{\mathcal E(\tau)}
	\leq
	\abs{v_{n_l}}_{\mathcal E(\tau_l)}
	\leq
	a_l
	\leq b
	<
	2\sqrt M .
	\]
	Letting $l\to\infty$ gives $\abs{v}_{\mathcal E(\tau)}
	\leq b < 2\sqrt M$ a.s.
	Consequently, since $M>1$,
	\[
	\mathbb E\left[
	\abs{v}_{\mathcal E(\tau)}^p
	\right]
	\leq
	(2\sqrt M)^p
	\leq
	C_p M^p,
	\]
	thus proving \eqref{equ:cauchy 4}.
	
	The sets $\Omega_N$ are increasing and satisfy
	$\mathbb P(\bigcup_N\Omega_N)=1$. Enlarging them by a null set if necessary,
	we may regard them as satisfying $\Omega_N\uparrow\Omega$. Upon relabelling the index $N$ by $l$, we use these sets as the sets $\Omega_l$ in the statement. Since on $\Omega_l$ we
	have $\tau\leq\tau_l$, it follows that
	\[
	\one_{\Omega_l}\abs{v_{n_l}}_{\mathcal E(\tau)}^p
	\leq
	\one_{\Omega_l}\abs{v_{n_l}}_{\mathcal E(\tau_l)}^p
	\leq
	b^p
	\leq
	C_p M^p,
	\]
	which implies \eqref{equ:cauchy 5} after taking the expectation.

	Finally, recalling the definition of $\Omega_l$ and by the triangle inequality, we have that, on $\Omega_l$,
	\[
	\abs{v_{n_l}-v}_{\mathcal E(\tau)}
	\leq
	\sum_{j=l}^{\infty}\varepsilon_j .
	\]
	Therefore, we infer that
	\[
	\mathbb E\left[
	\one_{\Omega_l}
	\abs{v_{n_l}-v}_{\mathcal E(\tau)}^p
	\right]
	\leq
	\left(
	\sum_{j=l}^{\infty}\varepsilon_j
	\right)^p
	\longrightarrow 0 \quad \text{as $l\to\infty$.}
	\]
	This shows \eqref{equ:lim cauchy 6}. The proof is now complete.
\end{proof}

\begin{lemma}[\cite{GlaZia09}]  \label{lem:weak conv}
Suppose that $X$ is a separable Banach space and let $Y\subset X$ be a dense subset. Let $X'$ be the dual of $X$ and denote the dual pairing between $u\in X$ and $v\in X'$ by $\inpro{u}{v}$. Assume that $(E,\mathcal{E},\mu)$ is a finite measure space. Let $p\in (1,\infty)$. Suppose that $\{v_n\}_{n\in\bb{N}}$ is a bounded sequence in $L^p(E,X')$ and 
\begin{equation}\label{equ:A2}
	\inpro{v_n}{y} \to \inpro{v}{y}, \; \text{ $\mu$-a.e. for all $y\in Y$.}
\end{equation}
Then $v_n \overset{\ast}{\rightharpoonup} v$ in $L^p(E,X')$, i.e., \eqref{equ:A2} holds for all $y\in X.$
\end{lemma}

\begin{lemma}[Poincar\'e-type inequalities]\label{lem:poincare}
	Let $\mathcal{H}$ and $\mathcal{A}$ be defined as in Subsection~\ref{subsec:fun set}. Let $\mathcal{H}_n:= \text{span}\{e_1,e_2,\ldots,e_n\}$, where $\{e_k\}_{k\in \bb{N}}$ is an orthonormal basis for $\mathcal{H}$ consisting of eigenfunctions of $\mathcal{A}$. Let $P_n:\mathcal{H}\to \mathcal{H}_n$ be the projection operator and let $Q_n=I-P_n$, where $I$ is the identity operator. Then for any real numbers $\alpha_1$ and $\alpha_2$, $\alpha_1<\alpha_2$, if $v\in \mathrm{D}(\mathcal{A}^{\alpha_2})$, then
	\begin{align}\label{equ:poincare Qn}
		\abs{Q_n v}_{\alpha_1} &\leq \lambda_n^{\alpha_1-\alpha_2} \abs{Q_n v}_{\alpha_2},
		\\
		\label{equ:Poincare Pn}
		\abs{P_n v}_{\alpha_2} &\leq \lambda_n^{\alpha_2-\alpha_1} \abs{P_n v}_{\alpha_1}.
	\end{align}
\end{lemma}

\begin{lemma}[Gronwall-type lemma for stochastic processes~\cite{GlaZia09}]\label{lem:gronwall}
Let $T>0$ be fixed. Suppose that $b,u,v,w: [0,T)\times \Omega\to\bb{R}$ are real-valued, non-negative stochastic processes. Let $\tau<T$ be a stopping time such that
\begin{align*}
	\bb{E}\left[ \int_0^\tau (b(t)u(t)+w(t))\, \dt \right] <\infty.
\end{align*}
Moreover, assume that for a fixed constant $\kappa$,
\begin{align*}
	\int_0^\tau b(t)\,\dt \leq \kappa, \; \text{ a.s.}
\end{align*}
Suppose that there exists a constant $C_0$ such that for all stopping times $\tau_a,\tau_b$ with $0\leq \tau_a < \tau_b \leq \tau$, we have
\begin{align*}
	\bb{E}\left[ \sup_{s\in [\tau_a,\tau_b]} u(s) + \int_{\tau_a}^{\tau_b} v(t)\,\dt\right] \leq C_0 \bb{E}\left[ u(\tau_a)+ \int_{\tau_a}^{\tau_b} (b(t)u(t)+w(t))\, \dt \right].
\end{align*}
Then
\begin{align*}
	\bb{E}\left[ \sup_{s\in [0,\tau]} u(s) + \int_0^\tau v(t)\,\dt\right] 
	\leq C\bb{E}\left[ u(0)+ \int_0^\tau w(t)\,\dt \right],
\end{align*}
where $C:= C(C_0,T,\kappa)$.
\end{lemma}

\section*{Statements and declarations}

\subsection*{Conflict of interest}
On behalf of all authors, the corresponding author states that there is no conflict of interest.

\subsection*{Funding statement}
The authors acknowledge financial support through the Australian Research Council's Discovery Projects funding scheme (projects DP200101866 and DP220101811).

Agus L. Soenjaya is supported by the Australian Government Research Training Program (RTP) Scholarship awarded at the University of New South Wales, Sydney.

\subsection*{Data availability}
No data are associated with this article.



\newcommand{\noopsort}[1]{}\def\cprime{$'$}
\def\soft#1{\leavevmode\setbox0=\hbox{h}\dimen7=\ht0\advance \dimen7
	by-1ex\relax\if t#1\relax\rlap{\raise.6\dimen7
		\hbox{\kern.3ex\char'47}}#1\relax\else\if T#1\relax
	\rlap{\raise.5\dimen7\hbox{\kern1.3ex\char'47}}#1\relax \else\if
	d#1\relax\rlap{\raise.5\dimen7\hbox{\kern.9ex \char'47}}#1\relax\else\if
	D#1\relax\rlap{\raise.5\dimen7 \hbox{\kern1.4ex\char'47}}#1\relax\else\if
	l#1\relax \rlap{\raise.5\dimen7\hbox{\kern.4ex\char'47}}#1\relax \else\if
	L#1\relax\rlap{\raise.5\dimen7\hbox{\kern.7ex
			\char'47}}#1\relax\else\message{accent \string\soft \space #1 not
		defined!}#1\relax\fi\fi\fi\fi\fi\fi}

\end{document}